\documentclass[11pt]{amsart}
\usepackage{amssymb,mathrsfs,graphicx,enumerate,mathabx}
\usepackage{amsmath,amsfonts,amssymb,amscd,amsthm,bbm}
\usepackage{mathtools}
\usepackage{colortbl}
\usepackage{graphicx, subfigure}
\usepackage[caption=false]{subfig}
\usepackage{graphicx}
\title[]{Ground states for aggregation-diffusion models on Cartan-Hadamard manifolds}

\author[Fetecau]{Razvan C. Fetecau}
\address[Razvan C. Fetecau]{\newline Department of Mathematics, Simon Fraser University, 8888 University Dr., Burnaby, BC V5A 1S6, Canada}
\email{van@math.sfu.ca}

\author[Park]{Hansol Park}
\address[Hansol Park]{\newline Department of Mathematics, Simon Fraser University, 8888 University Dr., Burnaby, BC V5A 1S6, Canada}
\email{hansol\_park@sfu.ca}

\newtheorem{theorem}{Theorem}[section]
\newtheorem{lemma}{Lemma}[section]
\newtheorem{corollary}{Corollary}[section]
\newtheorem{proposition}{Proposition}[section]
\newtheorem{remark}{Remark}[section]

\newcommand{\bbr}{\mathbb R}

\newcommand{\bbs}{\mathbb S}

\newcommand{\bbh}{\mathbb H}

\newcommand{\calA}{\mathcal{A}}
\newcommand{\calC}{\mathcal{C}}

\newcommand{\calP}{\mathcal{P}}
\newcommand{\calR}{\mathcal{R}}
\newcommand{\calK}{\mathcal{K}}

\newcommand{\calW}{\mathcal{W}}
\newcommand{\calJ}{\mathcal{J}}
\def\d{\mathrm{d}}

\newcommand{\dm}{n} 

\newcommand{\p}{o}
\newcommand{\h}{h}
\newcommand{\dist}{d}
\newcommand{\secc}{\calK}
\newcommand{\hc}{\phi} 

\makeatletter
\@namedef{subjclassname@2020}{\textup{2020} Mathematics Subject Classification}
\makeatother
\begin{document}

\date{\today}



\begin{abstract}
We consider a free energy functional on Cartan-Hadamard manifolds, and investigate the existence of its global minimizers. The energy functional consists of two components: an entropy (or internal energy) and an interaction energy modelled by an attractive potential. The two components have competing effects, as they favour spreading by linear diffusion and blow-up by nonlocal attractive interactions, respectively. We find necessary and sufficient conditions for existence of ground states for manifolds with sectional curvatures bounded above and below, respectively. In particular, for general Cartan-Hadamard manifolds, {\em superlinear} growth at infinity of the attractive potential prevents the spreading. The behaviour can be relaxed for homogeneous manifolds, for which only {\em linear} growth of the potential is sufficient for this purpose. 
\end{abstract}

\maketitle \centerline{\date}
\section{Introduction}
\label{sect:intro}
Consider a general Cartan-Hadamard manifold of finite dimension, an interaction potential $W: M \times M \to \bbr$, and denote by $\calP(M)$ the space of probability measures on $M$. Define the free energy functional $E: \calP(M) \to [-\infty,\infty]$ by
\begin{equation}
\label{eqn:energy-s}
E[\rho]=\int_M \rho(x)\log\rho(x)\d x+\frac{1}{2}\iint_{M \times M} W(x, y)\rho(x)\rho(y)\d x \d y,
\end{equation}
if $\rho$ is absolutely continuous with respect to the Riemannian volume measure $\d x$. Otherwise, $E[\rho]$ is defined by means of sequences $E[\rho_k]$, with $\rho_k$ absolutely continuous and $\rho_k \rightharpoonup \rho$ as $k \to \infty$ (see \eqref{eqn:energy} for the precise definition).

The energy $E$ relates to the following nonlinear evolution equation for a density $\rho$ on $M$:
\begin{equation}
\label{eqn:model}
\partial_t\rho(x)- \nabla_M \cdot(\rho(x)\nabla_M W*\rho(x))=\Delta \rho(x),
\end{equation}
where
\[
W*\rho(x)=\int_M W(x, y)\rho(y) \d y,
\]
and $\nabla_M \cdot$ and $\nabla_M $ represent the manifold divergence and gradient, respectively. Specifically, critical points of the energy $E$ are steady states of \eqref{eqn:model}, and also, equation \eqref{eqn:model} can be formally expressed as the gradient flow of the energy $E$ on a suitable Wasserstein space \cite{AGS2005}.

Free energy functionals and their associated evolution equations (known as {\em aggregation-diffusion equations}) have been extensively studied in the Euclidean case ($M=\bbr^\dm$). We refer to the excellent recent review \cite{CarrilloCraigYao2019}, which discusses a variety of issues for this class of equations, such as the well-posedness of solutions, existence/uniqueness of steady states, long-time behaviour of solutions and the metastability phenomenon. To mention only a few of these recent results, in \cite{CaDePa2019} the authors study the existence/non-existence of ground states of the free energy \eqref{eqn:energy} set up on $\bbr^\dm$, while in \cite{CaHiVoYa2019} it is shown that the stationary solutions of \eqref{eqn:model} are radially decreasing up to translation. The existence of minimizers for domains with boundaries was investigated in \cite{MeFe2020}. The long-time behaviour of solutions, in particular their equilibration toward the heat kernel in the diffusion dominated regime, was studied in \cite{CaGoYaZe}. The diffusive regularization and in particular, the metastability behaviour, was investigated in \cite{EvKo16}.  In the meanwhile, minimizers of the free energy in the absence of diffusion have been a separate subject of recent interest \cite{Balague_etalARMA, CaCaPa2015, ChFeTo2015, SiSlTo2015}.

In terms of applications, model \eqref{eqn:model} has been used in swarming and self-organized behaviour in biology \cite{CarrilloVecil2010, M&K, Morale:Capasso}, material science \cite{CaMcVi2006}, robotics \cite{JiEgerstedt2007}, and social sciences \cite{MotschTadmor2014}.  In such applications the potential $W$ in \eqref{eqn:model} can model interactions between biological organisms (e.g., insects, birds or cells), robots or even opinions. The manifold framework considered in this paper is particularly relevant to engineering applications, where agents/robots are typically restricted by environment or mobility constraints to remain on a certain manifold. Applications often use the discrete formulation of model \eqref{eqn:model}. Indeed, the continuum model \eqref{eqn:model} can be regarded as the mean-field limit of a coupled ODE system governing the evolution of $N$ interacting particles which undergo independent Brownian motions. The topic has extensive associated literature, we refer to \cite{Jabin2014} for a review of results in the Euclidean space. We also note that when $W$ is the attractive Newtonian potential in $\bbr^2$, \eqref{eqn:model} is the Patlak-Keller-Segel model of chemotaxis in mathematical biology. The model has an extensive literature on its own, and we refer again to \cite{CarrilloCraigYao2019} for a comprehensive literature review. Finally,  in two and three dimensions and $W$ being the Newtonian potential, \eqref{eqn:model} reduces to the Smoluchowski-Poisson equation in gravitational physics \cite{ChLaLe2004}. 

Alternatively, instead of linear diffusion one can consider nonlinear diffusion, where the entropy term in \eqref{eqn:energy-s} is replaced by $ \frac{1}{m-1} \int_M \rho^m(x)\d x$. For the evolution equation \eqref{eqn:model}, this results in the nonlinear diffusion term $\Delta \rho^m$. There has been extensive extensive work on aggregation models with repulsive/anticrowding effects modelled by nonlinear diffusion; e.g., see \cite{TBL} for an application in mathematical biology, and \cite{Bedrossian11,BurgerDiFrancescoFranek,BuFeHu14,CalvezCarrilloHoffmann2017,CaDePa2019,CaHoMaVo2018,CaHiVoYa2019,DelgadinoXukaiYao2022,Kaib17} for various results on the existence, uniqueness, and qualitative properties of the equilibria. We note that the two modes of diffusion (linear vs. nonlinear) result in different features of equilibria/minimizers of the associated energy. In particular, nonlinear diffusion models admit compactly supported equilibria,  in contrast with equilibria for linear diffusion which have full support \cite{CaDePa2019}.

In this paper we investigate the existence and non-existence of global minimizers of the energy functional $E[\rho]$ set up on general Cartan-Hadamard manifolds. In spite of the extensive research on the free energy \eqref{eqn:energy-s} in the Euclidean space $\bbr^\dm$, there is much less literature available for the manifold setup. In the absence of the interaction term, the energy functional \eqref{eqn:energy-s} was studied on general Riemannian manifolds \cite{CoMcSc2001, Sturm2005, RenesseSturm2005}, with the main focus being to characterize the geodesic convexity of the entropy functional in terms of the curvature of the manifold. On the other hand, there has been considerable interest lately on the interaction equation \eqref{eqn:model} without diffusion posed on Riemannian manifolds. We refer to \cite{CarrilloSlepcevWu2016,  FePa2021, FePaPa2020, PaSl2021, WuSlepcev2015} for well-posedness results, to \cite{FePa2023a} for a variational study on the hyperbolic space, and to \cite{FeZh2019, FePa2023b, HaChCh2014,  HaKaKi2022, HaKoRy2018} for the long-time behaviour of solutions. Rich pattern formation behaviours have been demonstrated numerically on various Riemannian manifolds such as the hyperbolic space \cite{FeZh2019, Ha-hyperboloid}, the special orthogonal group $SO(3)$ \cite{FeHaPa2021}, and Stiefel manifolds \cite{HaKaKi2022}.

To the best of our knowledge, the present work is the first that addresses the existence of ground states of free energies \eqref{eqn:energy-s} on general manifolds. Note that the energy \eqref{eqn:energy-s} has two components: the entropy (or internal energy) and the interaction energy modelled by the potential $W$. In the dynamic model \eqref{eqn:model}, the two parts result in the linear diffusion and the nonlocal transport term, respectively.  In the present work we are exclusively interested in purely {\em attractive} interaction potentials $W$ that are in the form $W(x,y) = h(\dist(x,y))$, for some non-decreasing function $h$, where $\dist$ denotes the geodesic distance on $M$. In such case there are competing effects in the energy from the local diffusion and the nonlocal attractive interactions. Nonexistence of energy minimizers can occur in two scenarios: i) spreading and ii) blow-up. The first happens when the nonlocal attraction is too weak at large distances and cannot prevent the spreading led by diffusion. The second scenario occurs when the attraction is too strong at short distances (singular at origin) and overcomes diffusion, resulting in blow-up. Consequently, the existence of ground states for the energy $E$ relies on a delicate balance of these two competing effects.
 
The main question that this paper addresses is the following. Let $M$ be a Cartan-Hadamard manifold, whose sectional curvatures $\calK$ are either bounded below or above, respectively, by two negative constants $-c_m$ and $-c_M$.  The question that we pose and answer is: {\em What is the behaviour of the attractive potential $W$, at infinity and at zero, that yields existence of global minimizers of the free energy?}. In the Euclidean space $M=\bbr^\dm$ this question was addressed in the work of Carrillo {\em et al.} \cite{CaDePa2019}. As shown in \cite[Theorem 6.1]{CaDePa2019} (see also \cite[Remark 6.2]{CaDePa2019}), a sharp condition for the existence of energy minimizers is that the attractive potential $W$ grows at least logarithmically at infinity (with a coefficient that depends on the dimension) to prevent the diffusion, and can have at most a logarithmic singularity at origin to avert the blow-up.

As a necessary condition, for manifolds with sectional curvatures bounded above by a negative constant, we show that unless $W$ grows at least linearly at infinity, nonexistence by spreading occurs. Regarding sufficient conditions for existence of ground states, we have two results. For general Cartan-Hadamard manifolds with sectional curvatures bounded below, we find that it is sufficient that $W$ grows at least {\em superlinearly} at infinity to prevent the spreading. For homogeneous manifolds however, we can refine this result and show that spreading is contained provided $W$ grows at least {\em linearly} at infinity,  with a coefficient that depends on the dimension and on the lower bound of the curvatures. This shows that the minimal linear growth we found for nonexistence is {\em sharp}. On the other hand, the singular behaviour at origin of the attractive potential, that is necessary and sufficient to avert the blow-up, is the same as in the Euclidean case (i.e., not worse than logarithmic). Indeed, curvature is not expected to play a significant role for the blow-up scenario, as in that case only the local behaviour (near origin) of the attraction matters.

The main tool we use to establish our results is a new logarithmic Hardy-Littlewood-Sobolev (HLS) inequality on Cartan-Hadamard manifolds which has an interest in its own. We derive this inequality from the logarithmic HLS inequality on $\bbr^\dm$ \cite{CaDePa2019, CarlenLoss1992} applied to the tangent space to $M$ at a fixed pole, along with the Rauch comparison theorem. Compared to the inequality on $\bbr^\dm$, the logarithmic HLS inequality on Cartan-Hadamard manifolds includes an additional term that contains the Jacobian of the exponential map at the pole. Since this Jacobian determines the volume growth of geodesic balls centred at the pole, we then use the bounds on sectional curvatures and volume comparison theorems in differential geometry to control the additional term in the HLS inequality on manifolds.

The paper is structured as follows. In Section \ref{sect:prelim}, we present some preliminaries, the assumptions and the main results. Section \ref{sec:diff-geom} contains some comparison results in differential geometry that are essential for our analysis. In Section \ref{sec:non-existence} we investigate conditions on the interaction potential for which ground states cannot exist (either by spreading or by blow-up); in particular, we provide the proof of Theorem \ref{thm:nonexist}.  Section \ref{sec:HLS} presents the logarithmic HLS inequality on Cartan-Hadamard manifolds.  In Section \ref{sec:existence-gen} we consider the case of general manifolds with sectional curvatures bounded from below, and prove Theorem \ref{thm:exist-gen} that establishes the existence of global minimizers. Section \ref{sec:existence}, which deals with existence of ground states in the case of homogeneous manifolds, contains the proof of Theorem \ref{them:existence}. Finally, the Appendix includes proofs of several results listed in the main body of the paper. 


\section{Assumptions and main results}
\label{sect:prelim}
\subsection{Assumptions, definitions and setup}
\label{subsect:notations}
The following assumptions on the manifold $M$ and interaction potential $W$ will be made throughout the entire paper.  
\smallskip

\noindent(\textbf{M}) $M$ is an $\dm$-dimensional Cartan-Hadamard manifold, i.e., $M$ is complete, simply connected, and has everywhere non-positive sectional curvature. We denote its intrinsic distance by $\dist$ and sectional curvature by $\calK$. In particular, for $x$ a point in $M$ and $\sigma$ a two-dimensional subspace of $T_x M$, $\calK(x;\sigma)$ denotes the sectional curvature of $\sigma$ at $x$.
\smallskip

\noindent(\textbf{W}) The interaction potential $W:M\times M\to\bbr$ has the form
\[
W(x, y)=\h(\dist(x, y)),\qquad \text{ for all }x, y\in M,
\]
where $\h:[0, \infty)\to [-\infty, \infty) $ is lower semi-continuous and non-decreasing. In particular, $h$ being non-decreasing indicates that the potential $W$ is purely {\em attractive}. Note that $h$ can have a singularity at origin (infinite attraction).

We also denote by $\d x$ the Riemannian volume measure on $M$ and by $\calP_{ac}(M) \subset \calP(M)$ the space of probability measures on $M$ that are absolutely continuous with respect to $\d x$. Throughout the paper, we will refer to an absolutely continuous measure directly by its density $\rho$, and by abuse of notation write $\rho \in \calP_{ac}(M)$ to mean $\d \rho (x) = \rho(x) \d x \in \calP_{ac}(M)$.
\smallskip

{\em Definition of the energy.} Let us make precise the definition of the energy introduced in \eqref{eqn:energy-s}. The energy functional 
$E: \calP(M) \to [-\infty, \infty]$ considered in this paper is given by
\begin{equation}
\label{eqn:energy}
E[\rho]=\begin{cases}
\displaystyle \int_M \rho(x)\log\rho(x) \d x+\frac{1}{2}\iint_{M \times M}  h(\dist(x, y))\rho(x)\rho(y) \d x\d y, & \text{ if  }\rho\in \mathcal{P}_{ac}(M),\\[2pt]
\displaystyle\inf_{\substack{\rho_k \in \calP_{ac}(M) \\ \rho_k \rightharpoonup \rho}} \liminf_{k\to \infty} E[\rho_k], & \text{ otherwise}.
\end{cases}
\end{equation}
Here, $\rho_k \rightharpoonup \rho$ denotes weak convergence as measures, i.e.,
\[
\int_M f(x) \rho_k(x) \d x \to \int_M f(x) \rho(x) \d x, \qquad \text{ as } k\to  \infty,
\]
for all bounded and continuous functions $f:M\to \bbr$.
\smallskip

{\em Invariance to isometries.} Both the entropy and the interaction energy are invariant under isometries of $M$. Indeed, assume $i: M \to M$ is an isometry on $M$, and consider a density $\rho\in \mathcal{P}_{ac}(M)$, and its pushforward (as measures) $i_\#\rho\in \mathcal{P}_{ac}(M)$ by $i$. Since an isometry is volume preserving, it holds that
\[
\rho(x) = i_\#\rho(i(x)),  \qquad \text{ for all } x\in M.
\]
Then, by the change of variable formula we have
\begin{align}
\int_M i_\#\rho(x)\log(i_\#\rho(x))\d x& = \int_M \rho(x)\log(i_\#\rho(i(x))) \d x \\
&= \int_M \rho(x)\log(\rho(x)) \d x,
\end{align}
which shows the invariance under isometries of the entropy. 

Regarding the interaction energy, by the change of variable formula and using that $i$ is distance preserving, we find
\begin{align*}
\iint_{M \times M}  h(\dist(x, y)) i_\# \rho(x) i_\# \rho(y) \d x\d y &= \iint_{M \times M}  h(\dist(i(x), i(y))) \rho(x) \rho(y) \d x\d y \\
& = \iint_{M \times M}  h(\dist(x,y)) \rho(x) \rho(y) \d x\d y.
\end{align*}
Since both of its components are invariant to isometries, so is the energy functional itself. Using the definition \eqref{eqn:energy} of the energy, one can also show this property for general densities $\rho \in \calP(M)$. Hence, for any isometry $i: M \to M$, it holds that
\begin{equation}
\label{eqn:inv-E}
E[i_\# \rho] = E[\rho], \qquad \text{ for any } \rho\in \calP(M).
\end{equation}

{\em Wasserstein spaces.} Denote by $\calP_1(M)$ the set of probability measures on $M$ with finite first moment, i.e.,
\[
\calP_1(M) = \left \{\mu \in \calP(M) : \int_M \dist(x,x_0) \d \mu(x) < \infty \right \},
\]
for some fixed (but arbitrary) point $x_0 \in M$. For $\rho,\sigma \in \calP(M)$, the \emph{intrinsic} $1$-Wasserstein distance is defined as
\begin{equation*}
	\calW_1(\rho,\sigma) = \inf_{\gamma \in \Pi(\rho,\sigma)} \iint_{M \times M} \dist(x,y) \d\gamma(x,y),
\end{equation*}
where $\Pi(\rho,\sigma) \subset \calP(M\times M)$ is the set of transport plans between $\rho$ and $\sigma$, i.e., the set of elements in $\calP(M\times M)$ with first and second marginals $\rho$ and $\sigma$, respectively. The space $(\calP_1(M),\calW_1)$ is a metric space. 
\smallskip

{\em Admissible densities.} To investigate the existence of energy minimizers, we fix the Riemannian centre of mass of the admissible densities. This is a natural assumption that is also made in the Euclidean case \cite{CaHiVoYa2019,CaHoMaVo2018,DelgadinoXukaiYao2022}. The Riemannian centre of mass (also known as the Karcher mean)  of a measure $\mu$  is a minimizer on $M$ of the function \cite{Karcher1977}:
\[
f(x) = \frac{1}{2} \int_M d^2(x,y) \d\mu(y).
\]
In particular, by the Euler-Lagrange equations, if $\mu$ has centre of mass $\p$, then 
\[
\int_M \log_\p y \, \d \mu(y) = 0,
\]
where $\log_\p$ denotes the Riemannian logarithm map at $\p$ (i.e., the inverse of the Riemannian exponential map $\exp_\p$) \cite{doCarmo1992}. It can be shown (we do not present these results here) that any $\mu \in \calP_1(M)$, where $M$ is a Cartan-Hadamard manifold, admits a unique centre of mass.  

We set up the following set of admissible densities. For a fixed arbitrary point $\p \in M$, we denote
\[
\mathcal{P}_\p(M):=\left\{\rho\in\mathcal{P}_{ac}(M) \cap \mathcal{P}_1(M): \int_M \log_\p x\rho(x)\d x=0\right\}.
\]
In this paper we fix $\p$ and establish existence of global minimizers of the energy functional in $\calP_\p(M)$. We note that in general, $\min_{\rho\in\mathcal{P}_{\p}(M)}E[\rho]$ is a function of the point $\p$. Indeed, for two distinct points $\p_1$ and $\p_2$, although global minimizers of $E[\cdot]$ exist in $\mathcal{P}_{\p_1}(M)$ and $\mathcal{P}_{\p_2}(M)$, respectively, there is no  guarantee that $\min_{\rho\in\mathcal{P}_{\p_1}(M)}E[\rho]$ and $\min_{\rho\in\mathcal{P}_{\p_2}(M)}E[\rho]$ are the same.

A manifold $M$ is said to be {\em homogeneous} if for any two points $\p$ and $\p'$ in $M$, there exists an isometry $i:M \to M$ such that $i(\p) = \p'$ \cite{doCarmo1992}. An example of a homogeneous Cartan-Hadamard manifold is the hyperbolic space. For homogeneous manifolds, $\min_{\rho\in\mathcal{P}_{\p}(M)}E[\rho]$ does not depend on the point $\p$. This is because the energy $E$ is invariant to isometries and hence, one can send the centre of mass of a minimizer to any point, via an isometry. In this case, solving the minimization problem on $\mathcal{P}_\p(M)$ and $\calP_{ac}(M) \cap \calP_1(M)$ are equivalent. 

\begin{remark}
\normalfont
We point out that we have two instances of the logarithm notation in this paper. One is the natural logarithm of a number (denoted with {\em no} subindex, as in $\log \rho(x)$) and the other is the Riemannian logarithm (denoted {\em with} a subindex, as in $\log_\p x$). The former is a real number, the latter is a vector in the tangent space to $M$ at a point. We believe that their meaning will be clear from the context, and no confusion should arise from the similar notations.
\end{remark}

\subsection{Main results} 
\label{subsect:results}
We list now the main results established in this paper. The following nonexistence and existence results apply to Cartan-Hadamard manifolds with sectional curvatures bounded above/below.
\begin{theorem}[Nonexistence by blow-up or spreading]
\label{thm:nonexist}
Let $M$ be an $\dm$-dimensional Cartan-Hadamard manifold with sectional curvatures that satisfy, for all $x \in M$ and all two-dimensional subspaces $\sigma \subset T_x M$, 
\[
\calK(x;\sigma) \leq -c_M<0, 
\] 
for some constant $c_M>0$. Also assume that $h:[0,\infty) \to [-\infty,\infty)$ is a non-decreasing lower semi-continuous function which satisfies
\[
\text{ either } \quad \limsup_{\theta\to0+} \left(h(\theta)-A_1\log\theta \right) < \infty \qquad\text{or}\quad  \limsup_{\theta\to\infty}\left(h(\theta)-A_2 \sqrt{c_M}\, \theta \right) < \infty,
\]
for some $A_1>2\dm$ or $A_2<\dm-1$. Then, the energy functional \eqref{eqn:energy} admits no global minimizer in $\calP_{ac}(M)$.
\end{theorem}

For general Cartan-Hadamard manifolds, the existence result established in this paper is the following.
\begin{theorem}[Existence of a global minimizer: general case]
\label{thm:exist-gen}
Let $M$ be an $\dm$-dimensional Cartan-Hadamard manifold with sectional curvatures that satisfy, for all $x \in M$ and all two-dimensional subspaces $\sigma \subset T_x M$, 
\begin{equation}
\label{eqn:K-lb-c}
-c_m  \leq \calK(x;\sigma)\leq 0, 
\end{equation}
for some constant $c_m>0$. Also assume that $h:[0, \infty)\to[-\infty, \infty)$ is a non-decreasing lower semi-continuous function which satisfies
\begin{equation}
\label{eqn:h-hyp}
\liminf_{\theta\to 0+}\left(h(\theta)-A_1\log\theta\right)>-\infty, \qquad \text{ and } \qquad \liminf_{\theta\to\infty}(h(\theta)-\hc(\theta)) >-\infty,
\end{equation}
for some $0<A_1<2 \dm$, where $\hc$ is a non-decreasing convex function with the property that
\begin{equation}
\label{eqn:hc}
\lim_{\theta\to\infty}\frac{\hc(\theta)}{\theta}=\infty.
\end{equation}
Then there exists a global minimum of $E[\rho]$ in $\calP_\p(M)$, for any $\p \in M$ fixed.
\end{theorem}

We note that the assumption \eqref{eqn:h-hyp} is satisfied by many interesting interaction potentials used in the literature, such as power-law $h(\theta) = \theta^\alpha$ with $\alpha>1$, and exponential $h(\theta) = e^\theta$. When the manifold is homogeneous, we can relax the assumption of superlinear growth at infinity and require only linear behaviour instead. The existence result we show in this case is the following.

\begin{theorem}[Existence of a global minimizer: homogeneous manifolds]
\label{them:existence}
Let $M$ be a homogeneous $\dm$-dimensional Cartan-Hadamard manifold, whose sectional curvatures satisfy, for all $x \in M$ and all two-dimensional subspaces $\sigma \subset T_x M$, 
\begin{equation}
\label{eqn:K-lowerb}
-c_m \leq \calK(x;\sigma)\leq 0, 
\end{equation}
for some constant $c_m>0$. Also assume that $h:[0, \infty)\to[-\infty, \infty)$ is a non-decreasing lower semi-continuous function which satisfies
\begin{equation}
\label{eqn:assumptions-h}
\liminf_{\theta\to 0+}\left(h(\theta)-A_1\log\theta\right)>-\infty, \qquad \text{ and } \qquad \liminf_{\theta\to\infty}\left(h(\theta)-A_2 \sqrt{c_m} \, \theta\right)>-\infty,
\end{equation}
for some $0<A_1<2\dm$ and $A_2>2(\dm-1)$. Then there exists a global minimum of $E[\rho]$ in $\calP_\p(M)$, for any $\p\in M$ fixed.
\end{theorem}

\begin{remark}(Comparison with the Euclidean case)
\label{rmk:compare-AA}
\normalfont
The analogous results in $\bbr^\dm$ are \cite[Theorem 4.1]{CaDePa2019} for nonexistence and \cite[Theorem 6.1]{CaDePa2019}  (see also \cite[Remark 6.2]{CaDePa2019}) for existence of minimizers. As far as the behaviour at origin is concerned (the $\theta \to 0+$ limit), there are no differences. In both cases, the interaction potential cannot have a singularity worse than $2\dm \log \theta$ at origin, or else the attraction is too strong and results in blow-up. Since this is a local behaviour, curvature plays no key role. 

On the other hand, the behaviour at infinity in the manifold case is different from that in $\bbr^\dm$. In $\bbr^n$, the potential $h(\theta)$ needs to grow faster than $2 \dm \log \theta$ at infinity to prevent the diffusion, while in Theorems \ref{thm:exist-gen} and \ref{them:existence} the growth is required to be superlinear and linear, respectively.  Also, in Theorem \ref{thm:nonexist} spreading occurs if the potential grows slower than linear, with a coefficient that depends on dimension and the upper bound of curvatures.
\end{remark}

\begin{remark}(Superlinear versus linear attraction at infinity)
\label{eqn:superlinear}
\normalfont
For general manifolds, one key reason why superlinear attraction at infinity is needed, is to ensure that minimizing sequences preserve the centre of mass (see Lemma \ref{lemma:CM}). For homogeneous manifolds, this condition is relaxed to linear attraction at infinity, with a coefficient that depends on both the dimension and the lower bound of the sectional curvatures. Intuitively, larger attraction at infinity is needed to contain diffusion on manifolds of larger (in magnitude) curvature. In the homogeneous case, minimizing sequences do not necessarily preserve the centre of mass, but this can be compensated by the existence of isometries on homogeneous manifolds.
\end{remark}


\section{Comparison results in differential geometry}\label{sec:diff-geom}
In this section we present some comparison results in Riemannian geometry, that are key to our analysis of existence of energy minimizers. Recall that $M$ is assumed to be a Cartan-Hadamard manifold, and hence, it has no conjugate points, it is diffeomorphic to $\bbr^\dm$ and the exponential map at any point is a global diffeomorphism.

\subsection{Rauch comparison theorem}
\label{subsect:Rauch}
One tool in our proofs is Rauch's comparison theorem, which enables us to compare lengths of curves on different manifolds. The result is the following \cite[Chapter 10, Proposition 2.5]{doCarmo1992}.

\begin{theorem}[Rauch comparison theorem]\label{RCT}
Let $M$ and $\tilde{M}$ be Riemannian manifolds and suppose that for all $p\in M$, $\tilde{p}\in \tilde{M}$, and $\sigma\subset T_pM$, $\tilde{\sigma}\subset T_{\tilde{p}}\tilde{M}$, the sectional curvatures $\secc$ and $\tilde{\secc}$ of $M$ and $\tilde{M}$, respectively, satisfy
\[
\tilde{\secc}(\tilde{p}; \tilde{\sigma})\geq \secc(p;\sigma).
\]
Let $p\in M$, $\tilde{p}\in\tilde{M}$ and fix a linear isometry $i:T_pM\to T_{\tilde{p}}\tilde{M}$. Let $r>0$ be such that the restriction ${\exp_p}_{|B_r(0)}$ is a diffeomorphism and ${\exp_{\tilde{p}}}_{|\tilde{B}_r(0)}$ is non-singular. Let $c:[0, a]\to\exp_p(B_r(0))\subset M$ be a differentiable curve and define a curve $\tilde{c}:[0, a]\to\exp_{\tilde{p}}(\tilde{B}_r(0))\subset\tilde{M}$ by
\[
\tilde{c}(s)=\exp_{\tilde{p}}\circ i\circ\exp^{-1}_p(c(s)),\qquad s\in[0, a].
\]  
Then the length of $c$ is greater or equal than the length of $\tilde{c}$.
\end{theorem}
We will use Theorem \ref{RCT} to compare lengths of curves on a general Cartan-Hadamard manifold $M$ with lengths of curves on spaces of constant curvature.


\subsection{Bounds on the Jacobian of the exponential map and volume comparison theorems.}
\label{subsect:bounds-J}
We will follow notations and the general framework from the book by Chavel \cite{Chavel2006}.  A Jacobi field along a unit-speed geodesic $\gamma$ is a differentiable vector field $Y$ along $\gamma$ that satisfies the Jacobi equation
\begin{equation}
\label{eqn:Jacobi}
\nabla_t^2 Y + \calR(\gamma',Y) \gamma' = 0,
\end{equation}
where $\nabla$ denotes the Levi-Civita connection on $M$ and $\calR$ represents the curvature tensor. Denote by $\calJ$ the vector space of Jacobi fields along $\gamma$ and by $\calJ^\perp$ the subspace of $\calJ$ consisting of the Jacobi fields orthogonal to $\gamma$. Jacobi fields are variation fields through geodesics. In informal terms, they measure the spread and the volume growth of a geodesic spray.

Let $x \in M$ and $\gamma$ be a unit-speed geodesic through $x$, $\gamma(0) = x$. Set an orthonormal basis $\{e_1, \cdots, e_\dm\}$ of  $T_x M$, with $e_\dm = \gamma'(0)$, and consider the Jacobi fields $J_j(t)$, $j=1,\dots, \dm-1$, along the geodesic $\gamma$ that satisfy
\[
J_j(0)=0\qquad\text{and}\qquad J_j'(0)=e_j, \qquad  j =1,\dots,\dm-1.
\]
Note that $J_j(t) \in \calJ^\perp$ for all $t \geq 0$, $j=1,\dots, \dm-1$. Denote by $\calA(t)$ the matrix with columns $J_j(t)$, $j = 1,\dots,\dm-1$, i.e., 
\[
\mathcal{A}(t)=[J_1(t) \;  J_2(t) \; \dots \; J_{\dm-1}(t)].
\]
We point out that $\calA(t)$ has been constructed for the fixed geodesic $\gamma$, so it depends on it. When there is potential for confusion in this regard, we will use the notation  $\calA(t; u)$, where $u = \gamma'(0)$.

By the properties of the Jacobi fields, we have
\[
J_j(t)=(d\exp_x)_{te_\dm}(te_j),\qquad j =1,\dots, \dm-1,
\]
and hence we can write
\begin{equation}
\label{eqn:calA-exp}
\mathcal{A}(t)=t[(d\exp_x)_{te_\dm}(e_1) \quad (d\exp_x)_{te_\dm}(e_2) \quad  \dots \quad  (d\exp_x)_{te_\dm}(e_{\dm-1})].
\end{equation}

Particularly relevant for the present work is the Jacobian $J(\exp_x)$ of the exponential map.  Since
\[
(d\exp_x)_{te_\dm}(e_\dm)=e_\dm,
\]
by using \eqref{eqn:calA-exp} we can express
\begin{equation}
\label{eqn:Jexp-A}
J(\exp_x)(te_\dm)=\det( (d\exp_x)_{te_\dm})=\det\begin{bmatrix}
\frac{\mathcal{A}(t)}{t}&0\\[5pt]
0&1
\end{bmatrix}=\det(\mathcal{A}(t)/t)=\frac{\det\mathcal{A}(t)}{t^{\dm-1}}.
\end{equation}

For manifolds with bounded sectional curvatures, the Jacobian $J(\exp_x)$ can be bounded below and above as follows (the results are presented in Chavel's book \cite{Chavel2006}, and we will simply list the results from there). 

\begin{theorem}(\cite[Theorems III.4.1 and III.4.3]{Chavel2006}) 
\label{lemma:Chavel-thms}
Suppose the sectional curvatures of $M$ satisfy, for all $x \in M$ and all two-dimensional subspaces $\sigma \subset T_x M$, 
\begin{equation}
-c_m\leq\calK(x;\sigma)\leq -c_M<0,
\label{eqn:const-bounds}
\end{equation}
where $c_m$ and $c_M$ are positive constants. Then, for any $x \in M$ and a unit-speed geodesic $\gamma$ through $x$, it holds that 
\[
\left(\frac{\sinh(\sqrt{c_M}t)}{\sqrt{c_M}}\right)^{\dm-1}\leq \det\mathcal{A}(t)\leq \left(\frac{\sinh(\sqrt{c_m}t)}{\sqrt{c_m}}\right)^{\dm-1}.
\]
Since the geodesic $\gamma$ is arbitrary, we also have (see \eqref{eqn:Jexp-A}):
\[
\left(
\frac{\sinh(\sqrt{c_M}\|u\|)}{\sqrt{c_M}\|u\|}
\right)^{\dm-1}\leq |J(\exp_x)(u)|\leq\left(
\frac{\sinh(\sqrt{c_m}\|u\|)}{\sqrt{c_m}\|u\|}
\right)^{\dm-1},
\]
for all $u \in T_x M$.
\end{theorem}

\begin{remark}
\normalfont
We note that for the upper bound on $\det\mathcal{A}(t)$ a weaker condition on curvature can be assumed, namely that the Ricci curvature along $\gamma$ is greater than or equal to $-(\dm-1) c_m$ (see Theorem III.4.3 in \cite{Chavel2006}). We do not consider this weaker assumption here.
\end{remark}

A point on a manifold is called a {\em pole} if the exponential map at that point is a global diffeomorphism. On Cartan-Hadamard manifolds, all points are poles. From here on, we fix a point $\p$ on $M$ which we will be referring to as the pole of the manifold. We also make the notation
\[
\theta_x := \dist(o,x), \qquad \text{ for all } x \in M.
\]

The bounds in Theorem \ref{lemma:Chavel-thms} result in bounds on the volume of geodesic balls in $M$. To explain these results we will use notation  $\calA(t; u)$ to indicate the dependence of $\calA$ on the unit tangent vector $u$. Denote by $\bbs^{\dm -1}$ the unit sphere in $\bbr^\dm$ and by $\omega(\dm)$ the volume of the unit ball in $\bbr^\dm$.  By using geodesic spherical coordinates centred at $\p$, the Riemannian measure on $M$ is given by (see \cite[Theorem III.3.1]{Chavel2006})
\begin{equation}
\label{eqn:vol-form}
\d x(\exp_\p(t u)) = \det\mathcal{A}(t;u) \d t \d \sigma(u),
\end{equation}
where $\d \sigma(u)$ denotes the Riemannian measure on $\bbs^{\dm -1}$. 

Denote by $B_\theta(\p)$ the open geodesic ball centred at $\p$ of radius $\theta$, and by $|B_\theta(\p)|$ its volume.
\begin{corollary}(\cite[Theorems III.4.2 and III.4.4]{Chavel2006})
\label{cor:AV-bounds}
Suppose the sectional curvatures of $M$ satisfy \eqref{eqn:const-bounds}. Then, 
\[
 \dm w(\dm)\int_0^\theta \left(\frac{\sinh(\sqrt{c_M}t)}{\sqrt{c_M}}\right)^{\dm-1} \d t \leq |B_{\theta}(\p)|\leq \dm w(\dm) \int_0 ^\theta \left(\frac{\sinh(\sqrt{c_m} t)}{\sqrt{c_m}}\right)^{\dm-1} \d t.
\]
\end{corollary}
\begin{proof} The proof follows immediately from Theorem \ref{lemma:Chavel-thms} together with \eqref{eqn:vol-form}, by writing
\[
| B_\theta(\p)| = \int_0^\theta \int_{\bbs^{\dm-1}} \det\mathcal{A}(t;u) \d \sigma(u) \d t.
\]
\end{proof}
We also note that the two bounds in Corollary \ref{cor:AV-bounds} represent the volume of the ball of radius $\theta$ in the hyperbolic space $\bbh^\dm$ of constant curvatures $-c_M$, and $-c_m$, respectively.


\section{Nonexistence of a global minimizer }
\label{sec:non-existence}

We will show that the energy functional \eqref{eqn:energy} does not admit global minimizers when either the diffusion or the attraction are too strong. In this section we  assume that the sectional curvatures satisfy, for all $x \in M$ and all sections $\sigma \subset T_x M$, 
\begin{equation}
\label{eqn:K-upperb}
\calK(x;\sigma)\leq -c_M \leq 0,
\end{equation}
for some constant $c_M$. 

Fix a pole $\p \in M$ and consider the following family of probability density functions $\rho_R$ that depend on $R>0$:
\begin{equation}
\label{eqn:rhoR}
\rho_{R}(x)=\begin{cases}
\displaystyle\frac{1}{|B_R(\p)|},&\quad\text{when }x\in B_R(\p),\\[10pt]
0,&\quad\text{otherwise}.
\end{cases}
\end{equation}
We will estimate $E[\rho_R]$ in two limits: $R \to 0^+$ (blow-up) and $R \to \infty$ (spreading). 

From \eqref{eqn:rhoR}, the entropy component of the energy $E[\rho_R]$ is 
\begin{equation}
\label{eqn:SrhoR}
\int_M \rho_{R}(x)\log\rho_{R}(x)\d x=-\log|B_R(\p)|.
\end{equation}
By the volume comparison theorem (see Corollary \ref{cor:AV-bounds} part i)), when $c_M>0$, the volume $|B_R(\p)|$ can be estimated as
\begin{equation}
\label{eqn:vol-lowb}
|B_R(\p)| \geq \dm w(\dm) \int_0^R \left(\frac{\sinh(\sqrt{c_M}\theta)}{\sqrt{c_M}}\right)^{\dm-1}\d \theta.
\end{equation}
Note that when $c_M=0$, \eqref{eqn:vol-lowb} reduces to
\begin{equation}
\label{eqn:vol-lowb-0}
|B_R(\p)| \geq \dm w(\dm) \int_0^R \theta^{\dm-1}\d \theta = w(\dm) R^\dm,
\end{equation}
which says that $|B_R(\p)|$ is bounded below by the volume of the ball of radius $R$ in $\bbr^\dm$. To estimate further  $|B_R(\p)|$ when $c_M >0$, we introduce the following lemma.

\begin{lemma}\label{SIEQ}
For any $0<\epsilon<1$, there exists $\beta_\epsilon>0$ such that
\[
\frac{\sinh x}{x}\geq \beta_\epsilon e^{x(1-\epsilon)}, \qquad\forall x>0.
\]
\end{lemma}
\begin{proof}
Fix $\epsilon\in (0, 1)$. 
Since
\[
\lim_{x\to0}\frac{\sinh x}{x e^{(1-\epsilon)x}}=1,\quad \text{ and } \quad \lim_{x\to\infty}\frac{\sinh x}{x e^{(1-\epsilon)x}}=\infty,
\]
there exists the global minimum $\beta_\epsilon>0$ of
\[
\frac{\sinh x}{x e^{(1-\epsilon)x}},\qquad \text{ for } x\in(0, \infty).
\]
This shows the desired result.
\end{proof}

Define $p_a(R)$ as 
\[
p_a(R)=e^{-aR}\int_0^R\theta^{\dm-1}e^{a\theta}\d \theta.
\]
With $\epsilon\in (0, 1)$ fixed, use \eqref{eqn:vol-lowb} and Lemma \ref{SIEQ}, to estimate:
\begin{align}
\label{eqn:estcMneq0}
|B_R(\p)|&
\geq \int_0^R \dm w(\dm) \beta_\epsilon^{\dm-1} \theta^{\dm-1} e^{\sqrt{c_M}(\dm-1)(1-\epsilon)\theta}\d\theta  \nonumber \\
&= \dm w(\dm) \beta_\epsilon^{\dm-1} e^{aR}p_a(R),
\end{align}
where $a=\sqrt{c_M}(\dm-1)(1-\epsilon)$. Note that we use \eqref{eqn:estcMneq0} only for $c_M>0$, as for $c_M=0$ we simply use \eqref{eqn:vol-lowb-0}. Therefore,  \eqref{eqn:SrhoR} together with \eqref{eqn:vol-lowb-0} and \eqref{eqn:estcMneq0}, yield
\begin{align}
\label{eqn:Sbound}
& \int_M \rho_R(x)\log\rho_R(x)\d x \leq  \\
& \hspace{1cm}
\begin{cases}
-\log(\dm w(\dm) \beta_\epsilon^{\dm-1} )-\sqrt{c_M}(\dm-1)(1-\epsilon)R-\log p_a(R), \qquad&\text{ when }c_M>0, \\[5pt]
-\log(w(\dm)) - \dm\log R, \qquad&\text{ when } c_M=0.
\end{cases}
\end{align}

On the other hand, since the support of $\rho_R$ is $B_R(\p)$, 
\[
\sup_{x, y\in\mathrm{supp}(\rho_R)}\dist(x, y)=2R,
\]
and this allows us to estimate the interaction component of the energy $E[\rho_R]$ as
\begin{equation}
\label{eqn:Ibound}
\frac{1}{2}\iint_{M \times M} h(\dist(x, y))\rho_R(x)\rho_R(y)\d x\d y\leq \frac{1}{2}\iint_{M \times M} h(2R)\rho_R(x)\rho_R(y)\d x \d y=\frac{h(2R)}{2}.
\end{equation}

Finally, combine \eqref{eqn:Sbound} and \eqref{eqn:Ibound} to get
\begin{equation}
\label{eqn:Ebound}
E[\rho_R] \leq 
\begin{cases}
\displaystyle -\log(\dm w(\dm) \beta_\epsilon^{\dm-1} )-\sqrt{c_M}(\dm-1)(1-\epsilon)R-\log p_a(R) + \frac{h(2R)}{2},  \qquad&\text{ when }c_M>0, \\[7pt]
\displaystyle -\log(w(\dm)) - \dm\log R + \frac{h(2R)}{2}, \qquad&\text{ when } c_M=0.
\end{cases}
\end{equation}
Note that \eqref{eqn:Ebound} holds for any $\epsilon\in (0, 1)$ fixed.

The goal is to investigate the limits $R\to 0^+$ and $R\to \infty$, respectively. If the right-hand-side of \eqref{eqn:Ebound} goes to $-\infty$ in either of these limits, then 
no global minimizers of the free energy exist.

\subsection{Behavior at zero}\label{subsec:nonexist-blowup}
This case deals with local behaviour, so curvature does not play an essential role.  For this reason, we assume for this part that $c_M = 0$, so this result also applies to the Euclidean case $M= \bbr^\dm$.  The result is the following.

\begin{proposition}[Nonexistence by blow-up]
\label{prop:blowup}
Assume the sectional curvatures of $M$ satisfy $\calK\leq0$ and $h$ is a non-decreasing lower semi-continuous function which satisfies 
\[
\limsup_{\theta\to0+} \left(h(\theta)-A_1\log\theta \right) < \infty,
\]
for some $A_1>2\dm$. Then, there exists no global minimizer of the energy functional \eqref{eqn:energy}.
\end{proposition}
\begin{proof}
Since the interaction potential is defined up to a constant, we can assume without loss of generality that 
\[
\limsup_{\theta\to0+} \left(h(\theta)-A_1\log\theta \right) \leq 0.
\]
Then, there exists $R_1>0$ such that
\[
h(\theta)\leq A_1\log \theta, \qquad\forall0<\theta\leq R_1.
\]
From \eqref{eqn:Ebound} (case $c_M=0$), we then get
\begin{align*}
E[\rho_R]&\leq -\log(w(\dm))-\dm\log R+\frac{A_1}{2}\log(2R)\\
&=-\log(w(\dm))+\frac{A_1}{2}\log 2+\left(\frac{A_1}{2}-\dm\right)\log R,
\qquad \forall R\in\left(0, R_1/2\right].
\end{align*}
Since $\frac{A_1}{2}-\dm>0$, we find
\[
\lim_{R\to0+}E[\rho_{R}]=-\infty.
\]
\end{proof}


\subsection{Behavior at infinity}\label{subsec:nonexist-spread}
We assume here that $c_M>0$. By L'H\^opital's rule,
\[
\lim_{R\to\infty}\frac{p_a(R)}{R^{\dm-1}}=\lim_{R\to \infty}\frac{\int_0^R \theta^{\dm-1}e^{a\theta}d\theta}{R^{\dm-1}e^{aR}}=\lim_{R\to \infty}\frac{R^{\dm-1}e^{aR}}{(\dm-1) R^{\dm-2}e^{aR}+aR^{\dm-1}e^{aR}}=\frac{1}{a}.
\]
Hence, $-\log p_a(R)$ in the right-hand-side of \eqref{eqn:Ebound} decays logarithmically to $-\infty$ as $R \to \infty$, and therefore it is subdominated by the linear decay of $-\sqrt{c_M}(\dm-1)(1-\epsilon)R$. Nonexistence of minimizers (by spreading) will occur if the attraction term is not strong enough at infinity to control this linear term. The precise result is given by the following proposition.

\begin{proposition}[Nonexistence by spreading]
\label{prop:spread}
Assume the sectional curvatures of $M$ satisfy $\calK\leq -c_M<0$ and $h$ is a non-decreasing lower semi-continuous function which satisfies 
\[
\limsup_{\theta\to\infty} \left(h(\theta)-A_2 \sqrt{c_M} \, \theta \right) < \infty,
\]
for some $A_2<\dm-1$. Then, there exists no global minimizer of the energy functional \eqref{eqn:energy}.
\end{proposition}
\begin{proof}
By adding a constant to the interaction potential, we can assume without loss of generality that
\[
\limsup_{\theta\to\infty} \left(h(\theta)-A_2 \sqrt{c_M} \, \theta \right) \leq 0.
\]
Hence, there exists $R_2\geq0$ such that 
\[
h(\theta)\leq A_2 \sqrt{c_M} \, \theta,\qquad \forall \theta\geq R_2.
\]
From \eqref{eqn:Ebound}, we then find
\begin{align}
\label{eqn:ER-spread}
E[\rho_R]&\leq  -\log(\dm w(\dm) \beta_\epsilon^{\dm-1} )-\sqrt{c_M}(\dm-1)(1-\epsilon)R-\log p_a(R) + A_2 \sqrt{c_M} R \nonumber \\[3pt]
& \leq -\log(\dm w(\dm) \beta_\epsilon^{\dm-1}) -\sqrt{c_M}((\dm-1)(1-\epsilon)-A_2)R, \qquad \forall R\geq R_2/2,
\end{align}
where for the second inequality we discarded the term $-\log p_a(R) \leq 0$ (for $R$ large).

Since $A_2< \dm-1 $, we can choose $0<\epsilon<1$ which satisfies
\[
(\dm-1)(1-\epsilon)-A_2>0.
\]
By \eqref{eqn:ER-spread}, this implies 
\[
\lim_{R\to\infty}E[\rho_R]=-\infty.
\]
\end{proof}

Theorem \ref{thm:nonexist} from Section \ref{sect:prelim} follows now from Propositions  \ref{prop:blowup} and \ref{prop:spread}.


\section{Logarithmic HLS inequality on Cartan-Hadamard manifolds}\label{sec:HLS}
In this section, we derive a logarithmic HLS inequality on Cartan-Hadamard manifolds, which is a key tool in our studies. The starting point is the logarithmic HLS inequality on $\bbr^\dm$, given by the following theorem.
\begin{theorem}[Logarithmic HLS inequality on $\bbr^\dm$ \cite{CaDePa2019}]
\label{thm:HLS-Rd}
Let $\rho\in\mathcal{P}_{ac}(\bbr^\dm)$ satisfy $\log(1+|\cdot|^2)\rho\in L^1(\bbr^\dm)$. Then there exists $C_0\in \bbr$ depending only on $\dm$, such that
\[
-\int_{\bbr^\dm}\int_{\bbr^\dm}\log(|x-y|)\rho(x)\rho(y)\d x \d y \leq \frac{1}{\dm}\int_{\bbr^\dm} \rho(x)\log\rho(x) \d x+C_0.
\]
\end{theorem}

We will generalize Theorem \ref{thm:HLS-Rd} to Cartan-Hadamard manifolds with sectional curvatures bounded from below. For this purpose, fix a pole $\p\in M$ and denote by $f:M\to T_\p M$ the Riemannian logarithm map at $\p$, i.e.,
\begin{equation}
\label{eqn:R-log}
f(x)=\log_\p x,\qquad \text{ for all }x\in M.
\end{equation}
The inverse map of $f$, $f^{-1}:T_\p M\to M$ is the Riemannian exponential map
\[
f^{-1}(u)=\exp_\p u, \qquad \text{ for all } u\in T_\p M.
\]
On Cartan-Hadamard manifolds, the exponential map is a global diffeomorphism.

Take a density $\rho\in \mathcal{P}_{ac}(M)$, and its pushforward (as measures) $f_\#\rho\in \mathcal{P}_{ac}(T_\p M)$ by $f$. Then,
\[
f_\#\rho(u)=\rho(f^{-1}(u))|J(f^{-1})(u)|, \qquad \text{ for all } u\in T_\p M,
\]
where $J(f^{-1})$ denotes the Jacobian of the map $f^{-1}$. In our case, $J(f^{-1})$ represents the Jacobian of the exponential map $\exp_\p$, which was discussed in Section \ref{subsect:bounds-J}.

Since $M$ has non-positive curvature, by Rauch Comparison Theorem (see Theorem \ref{RCT}) we have
\begin{equation}
\label{eqn:dist-ineq}
\dist(x, y)\geq |f(x)-f(y)|, \qquad \text{ for all } x,y \in M.
\end{equation}
Inequality \eqref{eqn:dist-ineq} can be justified as follows. Fix any two points $x,y \in M$. In the setup of Theorem \ref{RCT}, take $\tilde{M} = T_\p M$, $i$ to be the identity map, and $c$ to be the geodesic curve joining $x$ and $y$. The assumption on the curvatures $\secc$ and $\tilde{\secc}$ holds, as $\secc \leq 0 = \tilde{\secc}$. The curve $\tilde{c}$ constructed in Theorem \ref{RCT} joins $\log_\p x$ and $\log_\p y$, and hence, its length is greater than or equal to the length of the straight segment joining $\log_\p x$ and $\log_\p y$. Together with Theorem \ref{RCT}, we then infer:
\[
|\log_\p x - \log_\p y| \leq l(\tilde{c}) \leq l(c) = \dist(x,y).
\]
Note that \eqref{eqn:dist-ineq} holds indeed with equal sign for $M=\bbr^\dm$. 

By \eqref{eqn:dist-ineq}, we have
\begin{equation}
\label{eqn:HLS-deriv1}
-\int_M\int_M\log(\dist(x, y))\rho(x)\rho(y)\d x \d y\leq -\int_M\int_M\log(|f(x)-f(y)|)\rho(x)\rho(y)\d x \d y,
\end{equation}
since $\log(\cdot)$ is an increasing function on $(0,\infty)$. Then, using the definition of the pushforward measure and the logarithmic HLS inequality on $\bbr^\dm \cong T_\p M$, we get 
\begin{align}
\label{eqn:HLS-deriv2}
 -\int_M\int_M\log(|f(x)-f(y)|)\rho(x)\rho(y)\d x \d y &= -\int_{T_\p M}\int_{T_\p M}\log(|w-z|)f_\#\rho(w)f_\#\rho(z)\d w \d z \nonumber 
 \\[2pt]
 &\leq \frac{1}{\dm}\int_{T_\p M}f_\#\rho(z)\log(f_\#\rho(z))\d z+C_0,
\end{align}
where $C_0$ is a constant that depends on $\dm$. Now, use again the property of the pushforward measure, to find
\begin{align}
\int_{T_\p M}f_\#\rho(z)\log(f_\#\rho(z))\d z& = \int_M \rho(x)\log(f_\#\rho(f(x))) \d x \nonumber \\
&= \int_M \rho(x)\log(\rho(x)|J(f^{-1})(f(x))| )\d x \nonumber \\
&= \int_M \rho(x)\log\rho(x)\d x+\int_M \rho(x)\log|J(f^{-1})(f(x))| \d x.
\label{eqn:HLS-deriv3}
\end{align}

By combining \eqref{eqn:HLS-deriv1}, \eqref{eqn:HLS-deriv2} and \eqref{eqn:HLS-deriv3}, we obtain
\begin{align}
\label{eqn:HLS-CH}
\begin{aligned}
& -\int_M\int_M\log(\dist(x, y))\rho(x)\rho(y)\d x \d y  \leq \\
& \hspace{3.5cm} \frac{1}{\dm}\int_M \rho(x)\log\rho(x)\d x+\frac{1}{\dm}\int_M \rho(x)\log|J(f^{-1})(f(x))| \d x+C_0.
 \end{aligned}
\end{align}
Compared to the logarithmic HLS inequality on $\bbr^\dm$, inequality \eqref{eqn:HLS-CH} has an additional term on the right-hand-side, containing the Jacobian of the exponential map. Assuming lower bounds of the sectional curvatures, as in \eqref{eqn:const-bounds}, we can bound this Jacobian from above (see Theorem \ref{lemma:Chavel-thms}), and generalize Theorem \ref{thm:HLS-Rd} to Cartan-Hadamard manifolds as follows.

\begin{theorem}[Logarithmic HLS inequality on Cartan-Hadamard manifolds] \label{HLSub}
Let $M$ be an $\dm$-dimensional Cartan-Hadamard manifold, and $\p \in M$ be a fixed (arbitrary) pole. Also assume that the sectional curvatures of $M$ satisfy, for all $x \in M$ and all sections $\sigma \subset T_x M$, 
\begin{equation}
\label{eqn:K-lowerB}
-c_m \leq \calK(x;\sigma)\leq 0,
\end{equation}
where $c_m$ is a positive constant. Then, if $\rho\in \mathcal{P}_{ac}(M)$ and $\displaystyle  \int_M \log(1+|\theta_x|^2)\rho(x)\d x<\infty$, we have
\begin{equation}
\label{eqn:HLS-CH-I}
\begin{aligned}
&-\int_M\int_M\log(\dist(x, y))\rho(x)\rho(y)\d x \d y\leq \\
& \hspace{3.5cm}  \frac{1}{\dm}\int_M \rho(x)\log\rho(x)\d x+\frac{(\dm-1)}{\dm}\int_M \log\left ( \frac{\sinh(\sqrt{c_m}\theta_x)}{\sqrt{c_m}\theta_x} \right ) \rho(x) \d x+C_0,
\end{aligned}
\end{equation}
where $C_0\in\bbr$ is the constant depending only on $\dm$ introduced in Theorem \ref{thm:HLS-Rd}.
\end{theorem}
\begin{proof}
The proof follows from \eqref{eqn:HLS-CH} and the upper bound on $J(f^{-1})$ given by Theorem \ref{lemma:Chavel-thms}:
\[
|J(f^{-1})(f(x))|\leq  
\left( \frac{\sinh(\sqrt{c_m} \|f(x)\|}{\sqrt{c_m} \|f(x)\|} \right)^{\dm-1} = 
\left( \frac{\sinh(\sqrt{c_m} \theta_x)}{\sqrt{c_m} \theta_x} \right)^{\dm-1} .
\]

\end{proof}


\section{Existence of a global minimizer: general case}
\label{sec:existence-gen}
In this section we present the proof of Theorem \ref{thm:exist-gen}, which establishes the existence of a global minimizer of the energy functional on general Cartan-Hadamard manifolds. Throughout the section we assume that the sectional curvatures and the interaction potential satisfy \eqref{eqn:K-lb-c}, and \eqref{eqn:h-hyp}, respectively. In particular, $h$ has superlinear growth at infinity, as it grows faster than a non-decreasing convex function $\hc$ that satisfies \eqref{eqn:hc}.

We will present several key results leading to the proof of Theorem \ref{thm:exist-gen}. 
\begin{lemma}\label{Lem6.1}
Let $M$ be a Cartan-Hadamard manifold, and let $\p \in M$ be the Riemannian centre of mass of a density $\rho \in \calP_{ac}(M) \cap \calP_1(M)$, i.e., 
\begin{equation}
\label{eqn:fixed-cm}
\int_M \log_\p x\rho(x) \d x=0.
\end{equation}
Then we have
\[
(2-\sqrt{2})\calW_1(\rho, \delta_\p)\leq \iint_{M\times M} \dist(x, y)\rho(x)\rho(y)\d x\d y\leq 2\calW_1(\rho, \delta_\p).
\]
Here, $\delta_\p$ denotes the Dirac delta measure centred at the pole $\p$.
\end{lemma}
\begin{proof} The proof is based on Rauch comparison theorem and various simple calculations and estimates. See Appendix \ref{appendix:Lem6.1} for the details. Note that this result does not require any information on the curvature, and it applies to general CH manifolds.
\end{proof}

\begin{lemma}\label{Lem6.2}
Let $h$ be a lower semi-continuous function which satisfies \eqref{eqn:h-hyp}. Then, there exists $\epsilon >0$ and a constant $C$ such that
\[
h(\theta)-2 \dm \log\theta-2(\dm-1)  \log\left ( \frac{\sinh(\sqrt{c_m}\theta)}{\sqrt{c_m}\theta} \right ) -\epsilon\theta \geq C, \qquad \text{ for all } \theta\in(0, \infty).
\]
\end{lemma}
\begin{proof}
For the behaviour at infinity, first note that for all $\theta \geq 0$,
\begin{equation}
\label{eqn:est-sinh}
\begin{aligned}
\log\left(\frac{\sinh(\sqrt{c_m}\theta)}{\sqrt{c_m}\theta}\right) &=\sqrt{c_m}\theta+\log\left(\frac{1-e^{-2\sqrt{c_m}\theta}}{2\sqrt{c_m}\theta}\right) \\
& \leq  \sqrt{c_m}\theta,
\end{aligned}
\end{equation}
where for the inequality we used
\[
\frac{1-e^{-x}}{x} \leq 1, \qquad\forall x \geq 0.
\]
Since $h$ has superlinear growth at infinity, there exists $\epsilon>0$ such that 
\[
\liminf_{\theta\to \infty}\left(h(\theta)-2\dm\log\theta-2(\dm-1) \sqrt{c_m}\theta  -\epsilon\theta\right)>-\infty.
\]
Consequently, by \eqref{eqn:est-sinh} we can set
\begin{equation}
\label{eqn:C2}
C_2: = \liminf_{\theta\to \infty}\left(h(\theta)-2\dm\log\theta-2(\dm-1)  \log\left ( \frac{\sinh(\sqrt{c_m}\theta)}{\sqrt{c_m}\theta} \right )  -\epsilon\theta\right)>-\infty.
\end{equation}

For the behaviour near $0$, note that 
\[
-2\dm\log\theta > -A_1 \log\theta, \qquad \text{ for } 0<\theta<1.
\]
Hence, also using that $ \lim_{\theta\to0+} \log\left ( \frac{\sinh(\sqrt{c_m}\theta)}{\sqrt{c_m}\theta} \right ) =0$ and the first assumption on $h$ in \eqref{eqn:h-hyp}, we can define 
\begin{equation}
\label{eqn:C1}
\begin{aligned}
C_1:=\liminf_{\theta\to 0+}\left(h(\theta)-2\dm\log\theta-2(\dm-1) \log\left ( \frac{\sinh(\sqrt{c_m}\theta)}{\sqrt{c_m}\theta} \right ) -\epsilon\theta\right) > -\infty.
\end{aligned}
\end{equation}

From the definition of limit inferior, there exist $0<A_1<A_2$ such that
\begin{equation}
\label{eqn:A1A2}
\begin{aligned}
C_1-1&< h(\theta)-2\dm\log\theta-2(\dm-1) \log\left ( \frac{\sinh(\sqrt{c_m}\theta)}{\sqrt{c_m}\theta} \right ) -\epsilon\theta, \qquad\forall \theta\in (0, A_1),\\[2pt]
C_2-1&<h(\theta)-2\dm\log\theta-2(\dm-1)  \log\left ( \frac{\sinh(\sqrt{c_m}\theta)}{\sqrt{c_m}\theta} \right ) -\epsilon\theta, \qquad\forall \theta\in(A_2, \infty).
\end{aligned}
\end{equation}
Furthermore, use that $h$ is a lower semi-continuous function to get
\[
C_3\leq h(\theta) - 2\dm\log\theta-2(\dm-1) \log\left ( \frac{\sinh(\sqrt{c_m}\theta)}{\sqrt{c_m}\theta} \right ) -\epsilon\theta, \qquad\forall \theta\in[A_1, A_2].
\]
Finally, define 
\begin{equation}
\label{eqn:C}
C:=\min(C_1-1, C_2-1, C_3),
\end{equation}
to get the desired result.
\end{proof}

Lemmas \ref{Lem6.1} and \ref{Lem6.2} can be used to bound from below the energy of a density by its $\calW_1$ distance to the delta measure at the pole. The result is the following.
\begin{proposition}
\label{prop:E-W1}
Let $M$ and $h$ satisfy the assumptions in Theorem \ref{thm:exist-gen} and let $\p$ be a fixed pole in $M$. Then there exist constants $C_0$, $C$ and $\epsilon>0$ such that
\begin{equation}
\label{eqn:est-E-below}
E[\rho] \geq-C_0 \dm+\frac{C}{2}+\frac{(2-\sqrt{2})\epsilon}{2}\calW_1(\rho, \delta_\p), \qquad \text{ for all } \rho \in \calP_\p(M).
\end{equation}
Specifically, $C_0$ depends on $\dm$, while $C$ and $\epsilon$ depend on $h$, $\dm$ and $c_m$.
\end{proposition}
\begin{proof}
Since $M$ satisfies the assumptions in Theorem \ref{HLSub} with a constant function $-c_m$ as a lower bound, we can write \eqref{eqn:HLS-CH-I} for a generic pole $z$ as
\begin{align*}
&-\int_M\int_M\log(\dist(x, y))\rho(x)\rho(y)\d x \d y\leq \\
& \hspace{3cm}  \frac{1}{\dm}\int_M \rho(x)\log\rho(x)\d x+\frac{(\dm-1)}{\dm}\int_M  \log\left ( \frac{\sinh(\sqrt{c_m} d(x,z))}{\sqrt{c_m} d(x,z)} \right ) \rho(x) \d x+C_0,
\end{align*}
for any $\rho\in\mathcal{P}_{ac}(M) \cap \mathcal{P}_1(M)$. Then, multiply the inequality above by $\rho(z)$ and integrate with respect to $z$ to get
\begin{equation}
\label{eqn:ineq-key}
\begin{aligned}
&-\int_M\int_M\log(\dist(x, y))\rho(x)\rho(y)\d x \d y\leq \\
& \hspace{1.5cm}  \frac{1}{\dm}\int_M \rho(x)\log\rho(x)\d x+\frac{(\dm-1)}{\dm}\int_M \int_M  \log\left ( \frac{\sinh(\sqrt{c_m} d(x,z))}{\sqrt{c_m} d(x,z)} \right ) \rho(x) \rho(z) \d x\d z+C_0.
\end{aligned}
\end{equation}

We estimate the energy functional using \eqref{eqn:ineq-key} as follows:
\begin{equation}
\label{eqn:Erho-lb1}
\begin{aligned}
& E[\rho]=\int_M \rho(x)\log\rho(x) \d x+\frac{1}{2}\iint_{M \times M} h(\dist(x, y))\rho(x)\rho(y)\d x \d y\\
& \; =\int_M \rho(x)\log\rho(x) \d x+\iint_{M \times M} \left(\dm\log(\dist(x, y))+(\dm-1)\log \left ( \frac{\sinh(\sqrt{c_m} d(x,y))}{\sqrt{c_m} d(x,y)} \right )\right) \rho(x) \rho(y) \d x\d y\\
&\; \quad  +\frac{1}{2}\iint_{M \times M} \left(
h(\dist(x, y))-2\dm\log(\dist(x, y))-2(\dm-1)\log \left ( \frac{\sinh(\sqrt{c_m} d(x,y))}{\sqrt{c_m} d(x,y)} \right )\right) \rho(x) \rho(y) \d x\d y
\\[5pt]
& \; \geq-C_0 \dm \\[2pt]
& \; \quad +\frac{1}{2}\iint_{M \times M} \left(
h(\dist(x, y))-2\dm\log(\dist(x, y))-2(\dm-1)\log \left ( \frac{\sinh(\sqrt{c_m} d(x,y))}{\sqrt{c_m} d(x,y)} \right )\right) \rho(x) \rho(y) \d x\d y.
\end{aligned}
\end{equation}

Then, for any $\rho \in \mathcal{P}_\p(M)$, using Lemma \ref{Lem6.2} along with its notations (e.g., see \eqref{eqn:C}), we get
\begin{equation}
\label{ineq:E-seq}
\begin{aligned}
E[\rho]&\geq-C_0\dm+\frac{1}{2}\iint_{M\times M} (C+\epsilon \dist(x, y))\rho(x)\rho(y)\d x \d y \\
&=-C_0 \dm+\frac{C}{2}+\frac{\epsilon}{2}\iint_{M\times M} \dist(x, y)\rho(x)\rho(y)\d x \d y  \\[2pt]
&\geq-C_0 \dm+\frac{C}{2}+\frac{(2-\sqrt{2})\epsilon}{2}\calW_1(\rho, \delta_\p), 
\end{aligned}
\end{equation}
where we also used Lemma \ref{Lem6.1} in the second inequality.
\end{proof}

\begin{remark}
\label{rmk:int-HLS}
\normalfont
We will be referring to the inequality \eqref{eqn:ineq-key} as the {\em integrated logarithmic HLS inequality.} In certain instances, it is this weaker form of the HLS inequality that we use in our proofs.
\end{remark}

We now present several lemmas which will be used to establish the lower semi-continuity of the energy functional on minimizing sequences. 
\begin{lemma}\label{l33}
Let $M$ and $h$ satisfy the assumptions in Theorem \ref{thm:exist-gen}. Then for all $\rho\in \mathcal{P}_{ac}(M) \cap \calP_1(M)$, the energy functional can be bounded below as
\begin{align}\label{x-4}
E[\rho]\geq \delta \int_M\rho(x)\log\rho(x)\d x-C_0\dm(1-\delta)+\frac{\mathcal{C}}{2},
\end{align}
for some constants $\delta$ and $\calC$ that depend on $h$, $\dm$ and $c_m$, and $C_0$ is the constant from the HLS inequality \eqref{eqn:HLS-CH-I}. 
\end{lemma}
\begin{proof}
Using very similar considerations as in the proof of Lemma \ref{Lem6.2}, one can find a constant $\calC$ such that
\[
h(\theta)\geq A_1\log\theta+2(\dm-1)\log\left( \frac{\sinh(\sqrt{c_m}\theta)}{\sqrt{c_m}\theta}\right)+\mathcal{C}, \qquad \text{ for all } \theta>0.
\]
Define $0<\delta<1$ by $A_1=2\dm(1-\delta)$. Then,
\begin{align*}
& E[\rho] \geq \int_M \rho(x)\log\rho(x)\d x \\[2pt]
& \quad +\frac{1}{2}\iint_{M\times M} \left( 2 \dm(1-\delta) \log(\dist(x, y))+2(\dm-1)\log\left( \frac{\sinh(\sqrt{c_m}d(x,y))}{\sqrt{c_m}d(x,y)}\right)+\mathcal{C}\right)\rho(x)\rho(y) \d x \d y.
\end{align*}
Furthermore, after some simple manipulations, we find
\begin{align*}
E[\rho] & \geq \int_M \rho(x)\log\rho(x)\d x \\
& \quad + (1-\delta)\iint_{M \times M} \left(\dm\log(\dist(x, y))+(\dm-1)\log\left( \frac{\sinh(\sqrt{c_m}d(x,y))}{\sqrt{c_m}d(x,y)}\right)\right)\rho(x)\rho(y)\d x \d y+\frac{\mathcal{C}}{2}\\[2pt]
& \geq \delta\int\rho(x)\log\rho(x)\d x-C_0\dm(1-\delta)+\frac{\mathcal{C}}{2},
\end{align*}
where for the last line we used the integrated logarithmic HLS inequality \eqref{eqn:ineq-key}. This shows the claim.
\end{proof}

\begin{lemma} 
\label{lemma:unif-int}
Let $M$ be an $\dm$-dimensional Cartan-Hadamard manifold with sectional curvatures that satisfy \eqref{eqn:K-lb-c}. Fix a pole $\p \in M$, and suppose $\{\rho_k\}_{k\in\mathbb{N}} \subset \calP_{ac}(M)$ is a sequence of probability densities such that
\[
\int_M \theta_x \rho_k(x)\d x,\quad \text{ and } \quad \int_M \rho_k(x)\log\rho_k(x)\d x
\]
are bounded from above uniformly in $k\in\mathbb{N}$. Then $ \{\rho_k\}_{k\in\mathbb{N}}$ is uniformly integrable on $M$.
\end{lemma}
\begin{proof}
The proof is a little long and technical, we present it in Appendix \ref{appendix:unif-int}.
\end{proof}

\begin{lemma}
\label{lem:hc-rho}
Assume $M$ and $h$ satisfy the assumptions in Theorem \ref{thm:exist-gen}. Let $\p \in M$ be a fixed pole, and $\{\rho_k\}_{k\in\mathbb{N}} \subset \calP_{\p}(M)$ a sequence of probability densities such that $E[\rho_k]$ is uniformly bounded from above. Then
\[
\int_M \hc (\theta_x)\rho_k(x)\d x
\]
is also uniformly bounded from above, where $\hc$ is the function from assumption \eqref{eqn:h-hyp} on $h$.
\end{lemma}
\begin{proof}
From the definition of the energy functional, for any density $\rho \in \calP_{ac}(M)$ we have
\begin{align*}
& E[\rho]=\int_M \rho(x)\log\rho(x)\d x \\
& \quad +\iint_{M \times M} \left(\dm\log(\dist(x, y))+(\dm-1)\log\left( \frac{\sinh(\sqrt{c_m}d(x,y))}{\sqrt{c_m}d(x,y)}\right)\right)\rho(x)\rho(y)\d x \d y \\
&\quad +\frac{1}{2}\iint_{M\times M}\left(
h(\dist(x, y))-2\dm\log(\dist(x, y))-2(\dm-1)\log\left( \frac{\sinh(\sqrt{c_m}d(x,y))}{\sqrt{c_m}d(x,y)}\right)
\right)\rho(x)\rho(y)\d x\d y.
\end{align*}
Using \eqref{eqn:Erho-lb1} and $\log\left(\frac{\sinh\theta}{\theta}\right)\leq \theta$, one gets
\begin{equation}
\label{eqn:E-geq}
E[\rho]\geq -C_0\dm+\frac{1}{2}\iint_{M \times M} \left(h(\dist(x, y))-2\dm\log(\dist(x, y))-2(\dm-1)\sqrt{c_m}\dist(x, y)
\right)\rho(x)\rho(y)\d x\d y.
\end{equation}

We decompose
\begin{equation}
\label{eqn:decomp}
\begin{aligned}
&h(\theta)-2\dm\log\theta-2(\dm-1)\sqrt{c_m}\theta \\[2pt]
&\qquad =\left(h(\theta)-2\dm\log\theta-\frac{\hc(\theta)}{2}\right)+\left(\frac{\hc(\theta)}{4}-2(\dm-1)\sqrt{c_m}\theta\right)+\frac{\hc(\theta)}{4}.
\end{aligned}
\end{equation}
By combining \eqref{eqn:E-geq} and \eqref{eqn:decomp} we then find
\begin{equation}
\label{eqn:hcrho-bound}
\begin{aligned}
& \frac{1}{2}\iint_{M \times M}\frac{\hc(\dist(x, y))}{4}\rho(x)\rho(y)\d x \d y \leq E[\rho]+C_0\dm\\
&\hspace{4cm} -\frac{1}{2}\iint_{M \times M} \left(h(\dist(x, y))-2\dm\log\dist(x, y)-\frac{\hc(\dist(x, y))}{2}\right)\rho(x)\rho(y)\d x \d y\\
&\hspace{4cm} -\frac{1}{2}\iint_{M \times M} \left(\frac{\hc(\dist(x, y))}{4}-2(\dm-1)\sqrt{c_m}\dist(x, y)\right)\rho(x)\rho(y)\d x \d y,
\end{aligned}
\end{equation}
for all $\rho \in \calP_{ac}(M)$.

Arguing as in the proof of Lemma \ref{Lem6.2}, from
\[
\liminf_{\theta\to0+}\left(h(\theta)-2\dm\log\theta-\frac{\hc(\theta)}{2}\right) > -\infty,
\]
and
\[
\liminf_{\theta\to\infty}\left(h(\theta)-2\dm\log\theta-\frac{\hc(\theta)}{2}\right)\geq\liminf_{\theta\to\infty}\left(\frac{\hc(\theta)}{2}-2\dm\log\theta\right)=\infty,
\]
we infer that
\[
h(\theta)-2\dm\log\theta-\frac{\hc(\theta)}{2}
\]
has a lower bound on $(0,\infty)$. Similarly, from
\[
\lim_{\theta\to0+}\left(\frac{\hc(\theta)}{4}-2(\dm-1)\sqrt{c_m}\theta\right)=\frac{\hc(0)}{4}
\]
and
\[
\lim_{\theta\to\infty}\left(\frac{\hc(\theta)}{4}-2(\dm-1)\sqrt{c_m}\theta\right)=\infty,
\]
we also know that
\[
\frac{\hc(\theta)}{4}-2(\dm-1)\sqrt{c_m}\theta
\]
has a lower bound on $(0,\infty)$. 

Using these considerations in \eqref{eqn:hcrho-bound} for the sequence $\rho_k$ (note that $E[\rho_k]$ is uniformly bounded above) we infer that 
\[
\iint_{M \times M}\hc(\dist(x, y))\rho_k(x)\rho_k(y)\d x \d y
\]
is uniformly bounded in $k\in\mathbb{N}$. We will use this result to show the claim that the single integrals $\int_{M}\hc(\theta_x)\rho_k(x) \d x$ are uniformly bounded.

Since $\hc$ is non-decreasing, for any $x$ fixed we have
\begin{equation}
\label{eqn:hc-abs}
\int_M \hc(\dist(x, y))\rho_k(y)\d y\geq \int_M \hc(|\theta_x-\theta_y\cos\angle(x\p y)|)\rho_k(y)\d y, 
\end{equation}
where we used (see \eqref{eqn:dxy-ineq})
\[
\dist(x, y)^2\geq \theta_x^2+\theta_y^2-2\theta_x\theta_y\cos\angle(x\p y)\geq(\theta_x-\theta_y\cos\angle(x\p y))^2.
\]
Also, since $\hc$ is non-decreasing and convex, its even extension $\hc(|\cdot|)$ is also convex and we can apply Jensen's inequality to get
\begin{equation}
\label{eqn:hc-Jensen}
 \int_M \hc(|\theta_x-\theta_y\cos\angle(x\p y)|)\rho_k(y)\d y\geq \hc\left(\left|\int_M (\theta_x-\theta_y\cos\angle(x\p y))\rho_k(y)\d y\right|\right),
\end{equation}
for any $x$ fixed. 

As $\rho_k\in\mathcal{P}_\p(M)$, for any $x\neq \p$ we get
\[
0=\log_\p x\cdot\int_M \log_\p y\rho_k(y)\d y=\theta_x\int_M \theta_y\cos\angle(x\p y)\rho_k(y)\d y,
\]
which implies
\begin{equation}
\label{eqn:int-cos}
\int_M \theta_y\cos\angle(x\p y)\rho_k(y)\d y=0.
\end{equation}
Therefore,
\[
\int_M (\theta_x-\theta_y\cos\angle(x\p y))\rho_k(y)\d y=\theta_x,
\]
which used in \eqref{eqn:hc-Jensen}, and then combined with \eqref{eqn:hc-abs}, it leads to
\[
 \int_M \hc(\dist(x, y))\rho_k(y) \d y\geq \hc(\theta_x),
\]
for all $x$ fixed.

Finally, multiply by $\rho_k(x)$ both sides of the equation above and integrate with respect to $x$ to get
\[
\iint_{M \times M} \hc(\dist(x, y))\rho_k(x)\rho_k(y)\d x\d y \geq  \int_M \hc(\theta_x)\rho_k(x)\d x.
\]
The claim of the lemma now follows from the left-hand-side above being uniformly bounded in $k$.
\end{proof}

\begin{lemma}
\label{lem:Jint-cont}
Assume $M$ and $h$ satisfy the assumptions in Theorem \ref{thm:exist-gen}, and let $\p \in M$ be a fixed pole. Take a sequence $\{\rho_k\}_{k\in\mathbb{N}} \subset \calP_{\p}(M)$ such that $E[\rho_k]$ is uniformly bounded from above, and $\rho_k \rightharpoonup \rho_0$ weakly (as measures) as $k \to \infty$. Then we have
\[
\lim_{k\to\infty}\int_M \rho_k(x)\log|J(f^{-1})(f(x))|\d x=\int_M \rho_0(x)\log|J(f^{-1})(f(x))|\d x,
\]
where $f$ denotes the Riemannian logarithm map at $\p$ -- see \eqref{eqn:R-log}.
\end{lemma}
\begin{proof}
By Proposition \ref{prop:E-W1}, since $E[\rho_k]$ is uniformly bounded above, 
\[
\int_M \theta_x \rho_k(x) \d x = \calW_1(\rho_k,\delta_0) 
\]
is also uniformly bounded above. Consequently,  $\displaystyle \int_M\rho_k(x)\log|J(f^{-1})(f(x))|\d x$ is uniformly bounded above, as by Theorem \ref{lemma:Chavel-thms} and \eqref{eqn:est-sinh}, one has
\begin{equation}
\label{eqn:d-ineq}
\log|J(f^{-1})(f(x))|\leq (\dm-1)\log\left(\frac{\sinh(\sqrt{c_m}\theta_x)}{\sqrt{c_m}\theta_x}\right)\leq (\dm-1)\sqrt{c_m}\theta_x.
\end{equation}
We also note that 
\begin{equation}
\label{eqn:logJ-pos}
\log|J(f^{-1})(f(x))| \geq 0, \qquad \text{ for all } x \in M,
\end{equation}
as $ |J(f^{-1})(f(x))| \geq 1$ (this can be derived by taking $c_M \to 0^{-}$ in the lower bound on $|J|$ in Theorem \ref{lemma:Chavel-thms}).

Define
\[
U:=\sup_k \int_M\rho_k(x)\log|J(f^{-1})(f(x))|\d x <\infty.
\]
We first claim that 
\begin{equation}
\label{eqn:int-rho0}
\int_M\rho_0(x)\log|J(f^{-1})(f(x))|\d x< \infty.
\end{equation}
Indeed, assume by contradiction that $\displaystyle \int_M \rho_0(x)\log|J(f^{-1})(f(x))|\d x=\infty$. Then there exists $R>0$ such that
\begin{equation}
\label{eqn:int-rho0J}
\int_{\theta_x\leq R}\rho_0(x)\log|J(f^{-1})(f(x))|\d x\geq U+1.
\end{equation}
Since $\log|J(f^{-1})(f(x))|$ is continuous and bounded on the compact set $\{x\in M:\theta_x\leq R\}$, and $\rho_k$ converges to $\rho_0$ weakly as $k\to\infty$, we get
\begin{equation}
\label{eqn:pass-limit}
\lim_{k\to\infty}\int_{\theta_x\leq R}\rho_k(x)\log|J(f^{-1})(f(x))|\d x =\int_{\theta_x\leq R}\rho_0(x)\log|J(f^{-1})(f(x))|\d x. 
\end{equation}
Hence, by \eqref{eqn:logJ-pos}, \eqref{eqn:int-rho0J} and \eqref{eqn:pass-limit}, there exists $\tilde{k}$ such that
\begin{align*}
\int_{M}\rho_{\tilde{k}}(x)\log|J(f^{-1})(f(x))|\d x &\geq \int_{\theta_x\leq R}\rho_{\tilde{k}}(x)\log|J(f^{-1})(f(x))|\d x \\[2pt]
&\geq U+\frac{1}{2},
\end{align*}
but this contradicts the definition of $U$.

Fix $\epsilon>0$. By \eqref{eqn:int-rho0}, there exists $R_1>0$ such that
\begin{equation}
\label{eqn:rho0-eps}
\int_{\theta_x>R_1}\rho_0(x)\log|J(f^{-1})(f(x))|\d x<\epsilon.
\end{equation}
On the other hand, since $\hc$ satisfies
\[
\lim_{\theta\to\infty}\frac{\hc(\theta)}{\theta}=\infty,
\]
there exists $R_2>0$ such that
\begin{equation}
\label{eqn:l-eps}
0< \frac{\theta}{\hc(\theta)} <\epsilon, \qquad \text{ for all } \theta>R_2.
\end{equation}
Then, using \eqref{eqn:d-ineq} and \eqref{eqn:l-eps} we can derive
\begin{align*}
\int_{\theta_x>R_2}\rho_k(x)\log|J(f^{-1})(f(x))|\d x&\leq (\dm-1)\sqrt{c_m}\int_{\theta_x>R_2}\rho_k(x)\theta_x\d x\\ 
&= (\dm-1)\sqrt{c_m}\int_{\theta_x>R_2}\rho_k(x)\hc(\theta_x)\left(\frac{\theta_x}{\hc(\theta_x)}\right)\d x\\[2pt]
&\leq (\dm-1)\sqrt{c_m}\left\|\frac{\theta}{\hc(\theta)}\right\|_{L^\infty((R_2, \infty))}\int_{\theta_x > R_2} \rho_k(x) \hc(\theta_x)\d x\\ 
&\leq  (\dm-1)\sqrt{c_m}\epsilon\int_M \rho_k(x)\hc(\theta_x)\d x. 
\end{align*}
Since by Lemma \ref{lem:hc-rho}, $\displaystyle \int_M \hc(\theta_x)\rho_k(x)\d x$ is uniformly bounded in $k$ (call this uniform bound $\overline{U}$), from the above we get
\begin{equation}
\label{eqn:rhok-eps}
\int_{\theta_x>R_2}\rho_k(x)\log|J(f^{-1})(f(x))|\d x \leq (\dm-1)\sqrt{c_m}\,  \overline{U} \epsilon, \qquad \text{ for all } k \geq 1.
\end{equation}

Set $\tilde{R}:=\max(R_1, R_2)$; note that $\tilde{R}$ does not depend on $k$. Using the triangle inequality, equation \eqref{eqn:pass-limit} with $\tilde{R}$ in place of $R$, \eqref{eqn:rho0-eps} and \eqref{eqn:rhok-eps}, we get
\begin{align*}
&\limsup_{k\to\infty}\left|\int_M \rho_k(x)\log|J(f^{-1})(f(x))|\d x-\int_M \rho_0(x)\log|J(f^{-1})(f(x))|\d x\right|\\
& \quad \leq \limsup_{k\to\infty}\left|\int_{\theta_x\leq \tilde{R}}\rho_k(x)\log|J(f^{-1})(f(x))|\d x-\int_{\theta_x\leq \tilde{R}}\rho_0(x)\log|J(f^{-1})(f(x))|\d x\right|\\
&\qquad + \limsup_{k\to\infty}\left|\int_{\theta_x> \tilde{R}}\rho_k(x)\log|J(f^{-1})(f(x))|\d x\right|+ \limsup_{k \to \infty} \left|\int_{\theta_x> \tilde{R}}\rho_0(x)\log|J(f^{-1})(f(x))|\d x\right|\\[2pt]
&\quad \leq 0 +  (\dm-1)\sqrt{c_m}\overline{U} \epsilon +\epsilon.
\end{align*}
Since $\epsilon>0$ is arbitrary, we infer the claim of the lemma.
\end{proof}

\begin{remark}
\label{rmk:fmom-cont}
\normalfont
Using the same arguments as in the proof of Lemma \ref{lem:Jint-cont}, one can show that under the same assumptions, it also holds that
\[
\lim_{k\to\infty}\int_M \theta_x \rho_k(x)\d x=\int_M \theta_x \rho_0(x)\d x.
\]
Weak convergence together with convergence of first moments is equivalent to convergence in $\calW_1$ \cite{AGS2005}. This implies that for a sequence $\rho_k$ that satisfies the assumptions in Lemma \ref{lem:Jint-cont}, we also have
\[
\calW_1(\rho_k,\rho_0) \to 0, \qquad \text{ as } k \to \infty.
\]
\end{remark}

\begin{proposition}[Lower semi-continuity of the energy]
\label{prop:lsc}
Assume $M$ and $h$ satisfy the assumptions in Theorem \ref{thm:exist-gen}. Fix a pole $\p \in M$ and suppose $\{\rho_k\}_{k\in\mathbb{N}} \subset \calP_{\p}(M)$ is a sequence of probability densities such that $E[\rho_k]$ is uniformly bounded from above, and $\rho_k \rightharpoonup \rho_0$ weakly (as measures) as $k \to \infty$. Then, the energy $E$ is lower semi-continuous along $\rho_k$, i.e.,
\[
\liminf_{k\to\infty}E[\rho_k]\geq E[\rho_0].
\]
\end{proposition}
\begin{proof}
{\em \underline{Part 1.}}  {\em (Lower semi-continuity of the entropy)}

The lower semi-continuity of the entropy $\int_M \rho(x) \log \rho(x) \d x$  follows from the lower semi-continuity of the entropy in $\bbr^\dm$ \cite{CaDePa2019} and Lemma \ref{lem:Jint-cont}.  Indeed, from \eqref{eqn:HLS-deriv3}, we can write
\[
\int_M \rho_k(x)\log\rho_k(x)\d x=\int_{\bbr^\dm}f_\#\rho_k(u)\log f_\#\rho_k(u) \d u-\int_M \rho_k(x)\log |J(f^{-1})(f(x))|\d x,
\]
for all $k \geq 1$, where $f$ is the Riemannian logarithm map at $\p$.  Note that the first term in the right-hand-side above is the entropy of $f_\# \rho_k$ on $T_\p M\simeq \mathbb{R}^\dm$.  By the change of variable formula, it is immediate to show that $\rho_{k} \rightharpoonup \rho_0$ implies $f_\# \rho_{k} \rightharpoonup f_\# \rho_0$ as $k \to \infty$. Combine these observations with the lower semi-continuity of the entropy functional in $\bbr^\dm$ \cite{CaDePa2019} and Lemma \ref{lem:Jint-cont}, to get
\begin{equation}
\label{eqn:lsc-entropy}
\begin{aligned}
&  \liminf_{k \to\infty} \int_M \rho_{k}(x)\log\rho_{k}(x)\d x \\
& \qquad \geq  \liminf_{k \to \infty} \int_{\bbr^\dm}f_\#\rho_k(u)\log f_\#\rho_k(u) \d u- \lim_{k \to \infty} \int_M \rho_k(x)\log |J(f^{-1})(f(x))|\d x \\
& \qquad \geq  \int_{\bbr^\dm}f_\#\rho_0(u)\log f_\#\rho_0(u) \d u-\int_M \rho_0(x)\log |J(f^{-1})(f(x))|\d x\\
& \qquad = \int_M \rho_0(x)\log\rho_0(x)\d x,
 \end{aligned}
\end{equation}
where for the equal sign we used \eqref{eqn:HLS-deriv3} again.
\medskip

{\em \underline{Part 2.}}  {\em (Lower semi-continuity of the interaction energy)}

Define $C$ as
\[
C=\begin{cases}
0 \quad&\text{if}\quad \lim_{\theta\to 0+}h(\theta)\leq-1,\\
-\lim_{\theta\to 0+}h(\theta)-1&\text{otherwise}.
\end{cases}
\]
Note that $-\infty<C\leq 0$, as $\lim_{\theta\to 0+}h(\theta)<\infty$. In particular, if $h$ is singular at origin, i.e., $ \lim_{\theta\to 0+}h(\theta)=-\infty$, then $C=0$. 
\smallskip

{\em Behaviour near $0$.} By the definition of $C$, we have
\begin{equation}
\label{eqn:lim-hpc}
\lim_{\theta\to 0+}(h(\theta)+C)\leq-1.
\end{equation}
Hence, there exists $0<\theta_1<1$ such that
\begin{equation}
\label{eqn:hpC-neg}
h(\theta)+C<0, \qquad\forall 0<\theta<\theta_1.
\end{equation}
On the other hand, from the limiting behaviour of $h$ at $0$ (see \eqref{eqn:h-hyp}), one has
\[
\liminf_{\theta\to0+}\left(h(\theta)+C-2\dm\log\theta\right)=\infty.
\]

We then infer that there exists $0<\theta_2<\theta_1$ such that
\begin{equation*}
2\dm \log\theta \leq h(\theta)+C < 0, \qquad\forall 0<\theta<\theta_2.
\end{equation*}
Using this double inequality, for any $\delta< \theta_2$, we get
\begin{align}
\Biggl|\iint_{\dist(x, y)<\delta}(h(\dist(x, y))+C)\rho_k(x)\rho_k(y)\d x\d y\Biggr|
& \leq2\dm\Biggl |\iint_{\dist(x, y)<\delta}\log(\dist(x, y))\rho_k(x)\rho_k(y)\d x\d y\Biggr | \nonumber \\[2pt]
&= 2 \dm \iint_{\dist(x, y)<\delta}\log\left(\frac{1}{\dist(x, y)}\right)\rho_k(x)\rho_k(y)\d x\d y, \label{eqn:hpc-ineq}
\end{align}
where for the equal sign we used $\delta<1$. 

To estimate the right-hand-side in \eqref{eqn:hpc-ineq} we will use the following inequality:
\[
ab\leq e^a+b\log b-b,
\]
that holds for all $a,b \geq 0$ \cite{BlCaCa2008}.  Take $a=\beta\log\left(\frac{1}{\dist(x, y)}\right)$ and $b=\beta^{-1}\rho_k(y)$, with $\beta>0$ to be set later. Then, we have
\begin{equation}
\label{eqn:log-1od}
\begin{aligned}
& \iint_{\dist(x, y)<\delta}\log\left(\frac{1}{\dist(x, y)}\right)\rho_k(x)\rho_k(y)\d x\d y  \\
& \quad \leq \iint_{\dist(x, y)<\delta}\rho_k(x)\left(\frac{1}{\dist(x, y)^\beta}+\beta^{-1}\rho_k(y)\log(\beta^{-1}\rho_k(y))-\beta^{-1}\rho_k(y)\right)\d x\d y  \\
& \quad \leq \int_M \rho_k(x)\left(\int_{\dist(x, y)<\delta}\frac{\d y}{\dist(x, y)^\beta}\right)\d x+\beta^{-1}\int_M \rho_k(y)\log(\beta^{-1}\rho_k(y))\left(\int_{\dist(x, y)<\delta}\rho_k(x)\d x\right)\d y.
\end{aligned}
\end{equation}

If we choose $0<\beta< \dm$ then 
\[
\int_{\dist(x, y)<\delta}\frac{\d y}{\dist(x, y)^\beta}
\]
converges to zero as $\delta\to 0+$. We also note that since $E[\rho_k]$ is uniformly bounded above, so is $\displaystyle \int_M \rho_k(x) \log \rho_k(x) \d x$ (by Lemma \ref{l33}) and $\displaystyle \int_M \theta_x \rho_k(x) \d x$ (by Proposition  \ref{prop:E-W1}). Also, by Lemma \ref{lemma:unif-int}, $\{ \rho_k \}_{k \in \mathbb{N}}$ is uniformly integrable on $M$. 

Fix $\epsilon>0$ arbitrary small. By using the considerations above, and combining \eqref{eqn:hpc-ineq} with \eqref{eqn:log-1od}, we can conclude that for sufficiently small $\delta>0$, we can make
\begin{equation}
\label{hsmalldist}
\Biggl |\iint_{\dist(x, y)<\delta}(h(\dist(x, y))+C)\rho_k(x)\rho_k(y)\d x\d y\Biggr |<\epsilon, \qquad \text{ for all } k \geq 1.
\end{equation}
\smallskip

{\em Behaviour away from $0$.}  \underline{Claim 1}: For $\delta>0$ fixed (cf., \eqref{hsmalldist}),
\[
\iint_{\dist(x, y)\geq \delta}(h(\dist(x, y))+C)\rho_k(x)\rho_k(y)\d x\d y
\]
is uniformly bounded above in $k$, with an upper bound that does not depend on $\epsilon$.

This can be shown using $\log(\dist(x,y)) < \dist(x,y)$ in \eqref{eqn:E-geq} and then triangle inequality, by bounding the energy of $\rho_k$ as
\begin{align*}
E[\rho_k]& \geq -C_0\dm +\frac{1}{2}\iint_{M\times M}
\left( h(\dist(x, y))-2 \dm \dist(x, y)-2(\dm-1)\sqrt{c_m}\dist(x, y) \right)\rho_k(x)\rho_k(y)\d x\d y\\
&\geq -C_0\dm +\frac{1}{2}\iint_{M\times M}\left( h(\dist(x, y))-2\dm(\theta_x+\theta_y)-2(\dm-1)\sqrt{c_m}(\theta_x+\theta_y) \right)\rho_k(x)\rho_k(y)\d x\d y\\
&= -C_0\dm- 2(\dm+(\dm-1)\sqrt{c_m})\int_M \theta_x\rho_k (x)\d x+\frac{1}{2}\iint_{M \times M} h(\dist(x, y))\rho_k(x)\rho_k(y)\d x \d y.
\end{align*}
Then, use that $E[\rho_k]$ is uniformly bounded from above, along with the uniform boundedness of the first moments of $\rho_k$ (by Proposition \ref{prop:E-W1}) to infer that
\[
\iint_{M \times M} h(\dist(x, y))\rho_k(x)\rho_k(y)\d x\d y
\]
is uniformly bounded from above. This property together with $ \iint_{M \times M} C \rho_k(x)\rho_k(y)\d x \d y = C$ and \eqref{hsmalldist} (use $\epsilon <1$ for instance) lead to Claim 1. 

Set $U$ (independent of $\epsilon$) such that 
\[
\iint_{\dist(x, y)\geq \delta}(h(\dist(x, y))+C)\rho_k(x)\rho_k(y)\d x\d y \leq U, \qquad \text{ for all } k \geq 1.
\]
\smallskip

\underline{Claim 2}: For $\delta>0$ fixed as above,
\begin{equation}
\label{eqn:iint-ldelta}
\displaystyle  \iint_{\dist(x, y)\geq \delta}(h(\dist(x, y))+C)\rho_0(x)\rho_0(y)\d x\d y < \infty.
\end{equation}

To show this, assume by contradiction that
\begin{equation}
\label{eqn:int-inf}
\iint_{\dist(x, y)\geq \delta}(h(\dist(x, y))+C)\rho_0(x)\rho_0(y)\d x\d y=\infty.
\end{equation}
First note that since $\lim_{\theta\to\infty}(h(\theta)+C)=\infty$, there exists $\tilde{R}>0$ such that 
\begin{equation}
\label{eqn:tildeR}
h(\theta)+C>0 \qquad \text{ for all } \theta\geq\tilde{R}.
\end{equation}
From \eqref{eqn:int-inf}, there exists then $R> \tilde{R}$ such that
\begin{equation}
\label{eqn:tR}
\iint_{\delta\leq \dist(x, y)\leq R }(h(\dist(x, y))+C)\rho_0(x)\rho_0(y)\d x\d y\geq U+1.
\end{equation}

By \cite[Lemma 2.3]{BlCaCa2008}, the weak convergence of $\rho_k \rightharpoonup \rho$ implies that $\rho_k \otimes \rho_k\rightharpoonup \rho \otimes \rho$ weakly (as measures) as $k \to \infty$. Hence, we have
\begin{equation}
\label{eqn:lim-intdtR}
\lim_{k\to\infty}\iint_{\delta\leq \dist(x, y)\leq R}(h(\dist(x, y))+C)\rho_k(x)\rho_k(y)\d x\d y=
\iint_{\delta\leq \dist(x, y)\leq R}(h(\dist(x, y))+C)\rho_0(x)\rho_0(y)\d x\d y.
\end{equation}
From \eqref{eqn:tR} and \eqref{eqn:lim-intdtR}, there exists $\tilde{k}$ such that
\begin{align*}
\iint_{\delta\leq \dist(x, y)}(h(\dist(x, y))+C)\rho_{\tilde{k}}(x)\rho_{\tilde{k}}(y)\d x\d y &\geq \iint_{\delta\leq \dist(x, y)\leq R}(h(\dist(x, y))+C)\rho_{\tilde{k}}(x)\rho_{\tilde{k}}(y)\d x\d y \\
&>U + \frac{1}{2}.
\end{align*}
This contradicts the definition of $U$, showing Claim 2. 

We now combine the two behaviours (near $0$ and away from $0$) to show the lower semi-continuity of the interaction energy. From \eqref{eqn:iint-ldelta}, there exists $R_0>\tilde{R}$ such that
\begin{equation}
\label{eqn:iint-dgdelta}
0\leq\iint_{\dist(x, y)> R_0}(h(\dist(x, y))+C)\rho_0(x)\rho_0(y)\d x\d y<\epsilon.
\end{equation}
Recall that $h(\theta) + C <0$ for $0<\theta<\delta$ (see \eqref{eqn:lim-hpc}). Using this fact together with \eqref{eqn:iint-dgdelta} and the weak convergence $\rho_k \otimes \rho_k\rightharpoonup \rho \otimes \rho$, we estimate:
\begin{equation}
\label{eqn:iint-est1}
\begin{aligned}
&\iint_{M \times M} (h(\dist(x, y))+C)\rho_0(x)\rho_0(y)\d x\d y  \\
&\quad =\iint_{\dist(x, y)<\delta} (h(\dist(x, y))+C)\rho_0(x)\rho_0(y)\d x\d y+\iint_{\delta\leq \dist(x, y)\leq R_0} (h(\dist(x, y))+C)\rho_0(x)\rho_0(y)\d x\d y\\
&\qquad +\iint_{R_0<\dist(x, y)} (h(\dist(x, y))+C)\rho_0(x)\rho_0(y)\d x\d y  \\
&\quad \leq 0+\lim_{k\to\infty}\iint_{\delta\leq \dist(x, y)\leq R_0} (h(\dist(x, y))+C)\rho_k(x)\rho_k(y)\d x\d y+\epsilon.
\end{aligned}
\end{equation}

On the other hand, by \eqref{hsmalldist}, for all $k \geq 1$, 
\begin{equation}
\label{eqn:iint-est2}
\begin{aligned}
& \iint_{\delta\leq \dist(x, y)\leq R_0} (h(\dist(x, y))+C)\rho_k(x)\rho_k(y)\d x\d y \\
& \hspace{1cm}  = \iint_{\dist(x, y)\leq R_0} (h(\dist(x, y))+C)\rho_k(x)\rho_k(y)\d x\d y - \iint_{\dist(x, y) < \delta} (h(\dist(x, y))+C)\rho_k(x)\rho_k(y)\d x\d y\\
& \hspace{1cm} < \iint_{\dist(x, y) \leq R_0} (h(\dist(x, y))+C)\rho_k(x)\rho_k(y)\d x\d y + \epsilon.
\end{aligned}
\end{equation}
Combining \eqref{eqn:iint-est1} and \eqref{eqn:iint-est2}, we find
\begin{equation}
\label{eqn:hpC-rho0-ub}
\begin{aligned}
&\iint_{M \times M} (h(\dist(x, y))+C)\rho_0(x)\rho_0(y)\d x\d y\\[2pt]
&\hspace{1cm} \leq \liminf_{k\to\infty}\Biggl (\iint_{\dist(x, y)\leq R_0}(h(\dist(x, y))+C)\rho_k(x)\rho_k(y)\d x\d y + \epsilon \Biggr)+\epsilon\\
&\hspace{1cm} \leq \liminf_{k\to\infty}\iint_{M \times M} (h(\dist(x, y))+C)\rho_k(x)\rho_k(y)\d x\d y+2\epsilon,
\end{aligned}
\end{equation}
where for the second inequality we used $h(\theta)+C>0$ for any $\theta\geq R_0$ (as $R_0>\tilde{R}$).

Finally, since $\epsilon>0$ is arbitrary, we get
\[
\iint_{M \times M} (h(\dist(x, y))+C)\rho_0(x)\rho_0(y)\d x\d y\leq \liminf_{k\to\infty}\iint_{M \times M} (h(\dist(x, y))+C)\rho_k(x)\rho_k(y)\d x\d y,
\]
and from 
\[
\iint_{M \times M} C\rho_0(x)\rho_0(y)\d x\d y= \iint_{M \times M} C\rho_k(x)\rho_k(y)\d x\d y = C,
\]
we conclude
\[
\iint_{M \times M}  h(\dist(x, y))\rho_0(x)\rho_0(y)\d x\d y\leq \liminf_{k\to\infty}\iint_{M \times M} h(\dist(x, y))\rho_k(x)\rho_k(y)\d x\d y.
\]
\end{proof}

The last preliminary result we show concerns the conservation of the centre of mass. Regarding the notations, we make the following remark first. 
\begin{remark}
\label{rmk:balls}
\normalfont
We will concurrently use similar notations for geodesic open balls in $(M,\dist)$ and $(\calP_1(M),\calW_1)$. For instance, $B_\theta(\p)$ is the geodesic ball of radius $\theta$ and centre at $\p$, in the geodesic space $(M,\dist)$. On the other hand, $B_R(\delta_\p)$ denotes the geodesic ball in the space $(\calP_1(M),\calW_1)$, of radius $R$ and centre at $\delta_\p$.  In spite of similar notations, the space in which we consider geodesic balls will be clear from the context. 
\end{remark}

\begin{lemma}[Conservation of centre of mass]
\label{lemma:CM}
Let $M$ be a Cartan-Hadamard manifold, $\p \in M$ a fixed pole, $R>0$ a fixed radius,  and $\{\rho_k\}_{k\in\mathbb{N}}$ a sequence in $\overline{B_R(\delta_\p)}\cap \mathcal{P}_\p(M)$. Also assume that
\[
\int_M \hc(\theta_x) \rho_k(x) \d x
\]
is bounded above uniformly in $k \in \mathbb{N}$, where $\hc$ satisfies \eqref{eqn:hc}. Then, there exists a subsequence of $\{\rho_k\}_{k\in\mathbb{N}}$ which converges weakly as measures to $\rho_0\in \overline{B_R(\delta_\p)}\cap \mathcal{P}_\p(M)$.
\end{lemma}

\begin{proof}
Consider the open ball $B_R(\delta_\p)$ in $(\calP_1(M), \calW_1)$ centred at $\delta_\p$ and of radius $R>0$. Note that $B_R(\delta_\p)$ (and its closure in $\calW_1$) are tight. Indeed, for any $\rho \in B_R(\delta_\p)$,
\[
\int_M \theta_x\, \d\rho(x) = \calW_1(\rho,\delta_\p) < R,
\]
and on the other hand,
\[
\int_M \theta_x\, \d\rho(x)\geq \int_{\theta_x > L}\theta_x\, \d\rho(x)\geq L\int_{\theta_x >L}\d\rho(x),
\]
for any $L>0$. By combining the inequalities above we then get
\[
\int_{\theta_x >L}\d\rho(x) < \frac{R}{L}, \qquad \text{ for all } L>0.
\]
Hence, for any $\epsilon>0$, set $L=\frac{R}{\epsilon}$ and find
\[
\int_{\theta_x > \frac{R}{\epsilon}}\d\rho(x) <\epsilon,
\]
showing tightness of $B_R(\delta_0)$.

Since $\overline{B_R(\delta_\p)}$ is a tight and closed set in $\mathcal{P}_1(M)$, by Prokhorov's theorem, the sequence $\{\rho_k\}_{k \in\mathbb{N}}$ has a subsequence that converges weakly to $\rho_0\in\overline{B_R(\delta_\p)}$. It remains to show that $\rho_0\in\mathcal{P}_\p(M)$, i.e., $\rho_0$ has centre of mass at $\p$.

Fix an arbitrary unit tangent vector $v\in T_\p M$. As $\rho_0\in\overline{B_R(\delta_\p)}$, we have
\[
\left| \int_{M} \rho_0(x) \log_\p x\cdot v \, \d x \right|  \leq \|v\| \int_{M}\rho_0(x) \theta_x \d x\leq R.
\]
Fix $\epsilon>0$ arbitrary small. By the convergence of integral $\displaystyle \int_{M} \rho_0(x) \log_\p x\cdot v \, \d x$ (see above),  there exists $r_1>0$ such that
\begin{equation}
\label{eqn:iint-rho0-log}
\left |\int_{\theta_x> r}\rho_0(x)\log_\p x\cdot v \, \d x \right| < \epsilon, \qquad \text{ for any } r>r_1.
\end{equation}

On the other hand, from the assumption \eqref{eqn:hc} on $\hc$, there exists $r_2>0$ such that
\[
0< \frac{\theta}{\hc(\theta)} <\epsilon, \qquad \text{ for any } r>r_2. 
\]
Then, for any $r>r_2$, 
\begin{equation}
\label{eqn:iint-rhok-log}
\begin{aligned}
\left|\int_{\theta_x> r}\rho_k(x)\log_\p x\cdot v \, \d x\right|&\leq \int_{\theta_x> r }\rho_k(x)\theta_x \d x \\
& = \int_{\theta_x> r}\rho_k(x)\hc(\theta_x)\left(\frac{\theta_x}{\hc(\theta_x)}\right) \d x\\[2pt]
&\leq \left\|\frac{\theta}{\hc(\theta)}\right\|_{L^\infty((r, \infty))}\times \int_M \rho_k(x)\hc(\theta_x)\d x \\[2pt]
& \leq U \epsilon,
\end{aligned}
\end{equation}
where $U$ denotes a uniform upper bound of $\int_M \hc(\theta_x) \rho_k(x) \d x$. 

By combining \eqref{eqn:iint-rho0-log} and \eqref{eqn:iint-rhok-log}, for any $r>\max(r_1, r_2)$, we get
\begin{equation}
\label{eqn:iint-CMdiff}
\begin{aligned}
&\left|\int_{M}\rho_0(x)\log_\p x\cdot v\d x-\int_{M}\rho_k(x)\log_\p x\cdot v \d x\right|\\
& \qquad \leq \left|\int_{\theta_x\leq r }\rho_0(x)\log_\p x\cdot v\d x-\int_{\theta_x\leq r }\rho_k(x)\log_\p x\cdot v \d x\right|\\[2pt]
& \qquad \quad +\left|\int_{\theta_x> r}\rho_0(x)\log_\p x\cdot v\d x-\int_{\theta_x> r }\rho_k(x)\log_\p x\cdot v \d x\right|\\[2pt]
& \qquad \leq \left|\int_{\theta_x\leq r }\rho_0(x)\log_\p x\cdot v\d x-\int_{\theta_x\leq r }\rho_k(x)\log_\p x\cdot v \d x\right|+(U+1)\epsilon.
\end{aligned}
\end{equation}
Since $\{x:\theta_x\leq r\}$ is a bounded set, by continuity of the Riemannian logarithm we have
\[
\int_{\theta_x\leq r }\rho_0(x)\log_\p x\cdot v \, \d x=\lim_{k\to \infty} \int_{\theta_x\leq r }\rho_k(x)\log_\p x\cdot v \, \d x.
\]

Finally, letting $k\to\infty$ in \eqref{eqn:iint-CMdiff}, we find
\[
\limsup_{k\to\infty}\left|\int_{M}\rho_0(x)\log_\p x\cdot v\d x-\int_{M}\rho_k(x)\log_\p x\cdot v \d x\right|\leq (U+1)\epsilon,
\]
for any $\epsilon>0$. From here we infer that
\[
\lim_{k\to\infty}\left|\int_{M}\rho_0(x)\log_\p x\cdot v\d x-\int_{M}\rho_k(x)\log_\p x\cdot v \d x\right|=0,
\]
and hence,
\[
\lim_{k\to\infty}\int_{M}\rho_k(x)\log_\p x\cdot v \d x=\int_{M}\rho_0(x)\log_\p x\cdot v\d x.
\]

Since $\rho_k\in\mathcal{P}_\p(M)$, we have $\int_M\rho_k(x)\log_\p x \cdot v \, \d x=0$ for all $k \geq 1$, and consequently, 
\[
\int_M\rho_0(x)\log_\p x \cdot v \, \d x =0.
\]
Since the constant unit vector $v$ is arbitrary, we infer the conservation of the centre of mass property.
\end{proof}

We now present the proof of Theorem \ref{thm:exist-gen}, one of the main results of this research.
\medskip

{\em Proof of Theorem \ref{thm:exist-gen}.} We will present the proof in several steps.

{\em \underline{Step 1.}} Choose a probability density function $\tilde{\rho}\in \mathcal{P}_\p(M)$ which satisfies $E[\tilde{\rho}]<\infty$. For example, we can construct
\[
\tilde{\rho}=\exp_\p(\cdot)_{\#} \left(\frac{1}{w(\dm)}\mathbf{1}_{B_1(0)}\right),
\]
where $\frac{1}{w(\dm)}\mathbf{1}_{B_1(0)}$ is the uniform distribution on the unit ball in $T_\p M$ centered at the zero tangent vector. From the boundedness of the support of this distribution, we can easily check the boundedness of the energy. If $\tilde{\rho}$ is a global minimizer of $E[\cdot]$ on $\mathcal{P}_\p(M)$ then the proof is done. So we assume $\tilde{\rho}$ is not a global minimizer. i.e., 
\[
E[\tilde{\rho}]>\inf_{\rho\in \mathcal{P}_{\p}(M)}E[\rho].
\]

Consider the constants $C_0$, $C$ and $\epsilon$ from Proposition \ref{prop:E-W1}. Since $E[\tilde{\rho}] < \infty$, there exists $R>0$ such that
\begin{equation*}
E[\tilde{\rho}]\leq-C_0\dm+\frac{C}{2}+\left(\frac{(2-\sqrt{2})\epsilon}{2}\right) R.
\end{equation*}
Then, using Proposition \ref{prop:E-W1}, for any $\rho\in  \mathcal{P}_\p(M)\setminus \overline{B_{R}(\delta_\p)}$ it holds that
\begin{equation}
\label{eqn:Erho-geq}
\begin{aligned}
E[\rho]& \geq -C_0\dm+\frac{C}{2}+\frac{(2-\sqrt{2})\epsilon}{2}\calW_1(\rho, \delta_\p) \\
& >  -C_0 \dm+\frac{C}{2}+\left(\frac{(2-\sqrt{2})\epsilon}{2}\right)R \\
& \geq E[\tilde{\rho}].
\end{aligned}
\end{equation}

From \eqref{eqn:Erho-geq}, we infer
\[
 \inf_{\rho\in \mathcal{P}_\p(M)\backslash  \overline{B_{R}(\delta_\p)}}E[\rho]\geq E[\tilde{\rho}]> \inf_{\rho\in \mathcal{P}_\p(M)}E[\rho],
\]
which in turn implies
\[
E_0:=\inf_{\rho\in \mathcal{P}_\p(M)}E[\rho] = \inf_{\rho\in \overline{B_{R}(\delta_\p)}\cap\mathcal{P}_\p(M)}E[\rho].
\]
\smallskip

{\em \underline{Step 2.}}
Take a minimizing sequence $\{\rho_k\}_{k\in\mathbb{N}}$ of $E[\rho]$ on $\overline{B_R(\delta_\p)}\cap \mathcal{P}_\p(M)$, i.e.,
\[
\{\rho_k\}\subset \overline{B_R(\delta_\p)}\cap \mathcal{P}_\p(M), \qquad \lim_{k\to\infty}E[\rho_k]=E_0.
\]
Without loss of generality, we can assume that $E[\rho_k]$ is non-increasing. In particular, $E[\rho_k]$ is uniformly bounded from above. Hence Lemma \ref{lem:hc-rho} applies, and furthermore, by Lemma \ref{lemma:CM}, there exists a subsequence $\{\rho_{k_l}\}_{l\in\mathbb{N}}$ of $\{\rho_k\}_{k\in\mathbb{N}}$ which converges weakly as measures to $\rho_0\in \overline{B_R(\delta_\p)} \cap \calP_\p(M)$ as $l \to \infty$. From Proposition \ref{prop:lsc}, we also have
\begin{equation}
\label{eqn:limErho-kl}
\lim_{l\to\infty}E[\rho_{k_l}]\geq E[\rho_0].
\end{equation}

Using that $\rho_{k_l}$ is a minimizing sequence, along with \eqref{eqn:limErho-kl}, we then find
\[
E_0=\lim_{l\to\infty}E[\rho_{k_l}]\geq E[\rho_0] \geq \inf_{\rho\in \mathcal{P}_\p(M)}E[\rho]=E_0.
\]
Finally, we infer $E[\rho_0]= E_0$, and conclude that $\rho_0$ is a global energy minimizer of the energy in $\mathcal{P}_\p(M)$. 

\hspace {14cm} $\qed$


\section{Existence of a global minimizer for homogeneous manifolds}
\label{sec:existence}
In this section we present the proof of Theorem \ref{them:existence}, which establishes the existence of a global minimizer of the energy functional on homogeneous Cartan-Hadamard manifolds.  The overall strategy will be similar to that of the proof of Theorem \ref{thm:exist-gen}. 

\begin{lemma}\label{Lem6.2h}
Let $h$ be a lower semi-continuous function which satisfies \eqref{eqn:assumptions-h}.
Then, there exists $\epsilon >0$ and a constant $C$ such that
\[
h(\theta)-2 \dm \log\theta-2(\dm-1)  \log\left ( \frac{\sinh(\sqrt{c_m}\theta)}{\sqrt{c_m}\theta} \right ) -\epsilon\theta \geq C, \qquad \text{ for all } \theta\in(0, \infty).
\]
\end{lemma}
\begin{proof}
For the behaviour at infinity, use \eqref{eqn:est-sinh}. Set
\begin{equation}
\label{eqn:epsilon}
\epsilon=(A_2- 2(\dm -1))\sqrt{c_m}>0.
\end{equation}
Then, by \eqref{eqn:est-sinh}, we get
\begin{align*}
h(\theta)-2 \dm \log\theta-2(\dm-1)  \log\left ( \frac{\sinh(\sqrt{c_m}\theta)}{\sqrt{c_m}\theta} \right )  -\epsilon\theta &\geq h(\theta)-2 \dm \log\theta-2(\dm-1) \sqrt{c_m} \theta  -\epsilon\theta  \\
&= h(\theta)-A_2 \sqrt{c_m} \theta - 2 \dm \log\theta.
\end{align*}
By the second assumption on $h$ in \eqref{eqn:assumptions-h}, this implies that we can define
\begin{equation}
\label{eqn:C2}
C_2: = \liminf_{\theta\to \infty}\left(h(\theta)-2\dm\log\theta-2(\dm-1)  \log\left ( \frac{\sinh(\sqrt{c_m}\theta)}{\sqrt{c_m}\theta} \right )  -\epsilon\theta\right)>-\infty.
\end{equation}

For the behaviour near $0$, note that 
\[
-2\dm\log\theta > -A_1 \log\theta, \qquad \text{ for } 0<\theta<1.
\]
Hence, also using that $ \lim_{\theta\to0+} \log\left ( \frac{\sinh(\sqrt{c_m}\theta)}{\sqrt{c_m}\theta} \right ) =0$ and the first assumption on $h$ in \eqref{eqn:assumptions-h}, we can define 
\begin{equation}
\label{eqn:C1}
\begin{aligned}
C_1:=\liminf_{\theta\to 0+}\left(h(\theta)-2\dm\log\theta-2(\dm-1) \log\left ( \frac{\sinh(\sqrt{c_m}\theta)}{\sqrt{c_m}\theta} \right ) -\epsilon\theta\right) > -\infty.
\end{aligned}
\end{equation}

From the definition of limit inferior, there exist $0<A_1<A_2$ such that
\begin{equation}
\label{eqn:A1A2}
\begin{aligned}
C_1-1&< h(\theta)-2\dm\log\theta-2(\dm-1) \log\left ( \frac{\sinh(\sqrt{c_m}\theta)}{\sqrt{c_m}\theta} \right ) -\epsilon\theta, \qquad\forall \theta\in (0, A_1),\\[2pt]
C_2-1&<h(\theta)-2\dm\log\theta-2(\dm-1)  \log\left ( \frac{\sinh(\sqrt{c_m}\theta)}{\sqrt{c_m}\theta} \right ) -\epsilon\theta, \qquad\forall \theta\in(A_2, \infty).
\end{aligned}
\end{equation}
Furthermore, use that $h$ is a lower semi-continuous function to get
\[
C_3\leq h(\theta) - 2\dm\log\theta-2(\dm-1) \log\left ( \frac{\sinh(\sqrt{c_m}\theta)}{\sqrt{c_m}\theta} \right ) -\epsilon\theta, \qquad\forall \theta\in[A_1, A_2].
\]
Finally, define 
\begin{equation}
\label{eqn:Ch}
C:=\min(C_1-1, C_2-1, C_3),
\end{equation}
to get the desired result.
\end{proof}

\begin{remark}
\label{rmk:transfer}
\normalfont
There are certain results from Section \ref{sec:existence-gen} that transfer immediately to the context of the current section. First, Lemma \ref{Lem6.1} does not even contain information about $h$, it is a general fact about densities on $M$. Second, having established Lemma \ref{Lem6.2h}, Proposition  \ref{prop:E-W1} can be proved in exactly the same way.  Finally, the proof of Lemma \ref{l33} follows identically as well, if $h$ satisfies the weaker condition \eqref{eqn:assumptions-h}. So we will adopt both Proposition  \ref{prop:E-W1} and Lemma \ref{l33}, with the convention that referring to one of these results in the current section we mean that in their assumptions, $h$ satisfies \eqref{eqn:assumptions-h} instead of \eqref{eqn:h-hyp}. 
\end{remark}

Next, we state and prove the lower semi-continuity of the energy along minimizing sequences.

\begin{proposition}[Lower semi-continuity of the energy]
\label{prop:lsc-homog}
Assume $M$ and $h$ satisfy the assumptions in Theorem \ref{them:existence}. Fix a pole $\p \in M$, and suppose $\{\rho_k\}_{k\in\mathbb{N}} \subset \calP_\p(M)$ is such that $E[\rho_k]$ is uniformly bounded from above, and $\rho_k \rightharpoonup \rho_0$ weakly (as measures) as $k \to \infty$. Then we have
\[
\liminf_{k\to\infty}E[\rho_k]\geq E[\rho_0].
\]
\end{proposition}
\begin{proof}
By \eqref{eqn:HLS-deriv3}, the entropy component of the energy $E[\rho_k]$ can be written as
\[
\int_M \rho_k(x)\log\rho_k(x)\d x=\int_{T_\p M}f_\#\rho_k(u)\log f_\#\rho_k(u) \d u-\int_M \rho_k(x)\log |J(f^{-1})(f(x))|\d x,
\]
where $f$ denotes the Riemannian logarithm map at $\p$. Using this equation, and adding and subtracting $\sqrt{c_m}(\dm-1)\iint_{M\times M} \dist(x, y)\rho_k(x)\rho_k(y)\d x \d y$ to $E[\rho_k]$, we express the energy functional as
\begin{equation}
\label{eqn:3-terms}
\begin{aligned}
E[\rho_k]&=\int_{T_\p M}f_\#\rho_k(u)\log f_\#\rho_k(u) \d u \\[2pt]
&+\frac{1}{2}\iint_{M\times M}  \left( h(\dist(x, y))-2(\dm-1)\sqrt{c_m}\dist(x, y) \right)\rho_k(x)\rho_k(y)\d x \d y\\
&+\Biggl (\sqrt{c_m}(\dm-1)\iint_{M\times M} \dist(x, y)\rho_k(x)\rho_k(y)\d x \d y-\int_M \rho_k(x)\log |J(f^{-1})(f(x))|\d x\Biggr ).
\end{aligned}
\end{equation}
We will investigate the each term of the right-hand-side of \eqref{eqn:3-terms}.
\medskip

{\em \underline{Part 1.}}  For the first term in the right-hand-side of \eqref{eqn:3-terms} we use the lower semi-continuity of the entropy functional in $\bbr^\dm$ \cite{CaDePa2019}.  Indeed, since by the change of variable formula, $\rho_{k} \rightharpoonup \rho_0$ implies $f_\# \rho_{k} \rightharpoonup f_\# \rho_0$ as $k \to \infty$, this property implies
\begin{equation}
\label{eqn:liminf-part1}
\begin{aligned}
\liminf_{k \to\infty} \int_M \rho_{k}(x)\log\rho_{k}(x)\d x &= \liminf_{k \to \infty} \int_{T_\p M}f_\#\rho_k(u)\log f_\#\rho_k(u) \d u \\[2pt]
& \geq  \int_{T_\p M}f_\#\rho_0(u)\log f_\#\rho_0(u) \d u.
\end{aligned}
\end{equation}
\medskip

{\em \underline{Part 2.}}  We will show
\begin{align}
\label{eqn:liminf-part2}
&\liminf_{k\to\infty}\iint_{M\times M} (h(\dist(x, y))-2(\dm-1)\sqrt{c_m}\dist(x, y))\rho_k(x)\rho_k(y)\d x \d y\\
& \hspace{3cm} \geq\iint_{M\times M} (h(\dist(x, y))-2(\dm-1)\sqrt{c_m}\dist(x, y))\rho_0(x)\rho_0(y)\d x \d y.
\end{align}
The steps to show \eqref{eqn:liminf-part2} are exactly the same as in Part 2 of the proof of Proposition \ref{prop:lsc}, which established the lower semi-continuity of the interaction energy. Since the proof of \eqref{eqn:liminf-part2} is very close conceptually to these prior arguments, we present it in Appendix \ref{appendix:part2}.
\medskip

{\em \underline{Part 3.}}  The goal is to show
\begin{equation}
\label{eqn:liminf-part3}
\begin{aligned}
&\liminf_{k\to\infty}\left(\sqrt{c_m}(\dm-1)\iint_{M\times M} \dist(x, y)\rho_k(x)\rho_k(y)\d x \d y-\int_M \rho_k(x)\log |J(f^{-1})(f(x))|\d x\right)\\
& \hspace{2cm} \geq \sqrt{c_m}(\dm-1)\iint_{M\times M} \dist(x, y)\rho_0(x)\rho_0(y)\d x \d y-\int_M \rho_0(x)\log |J(f^{-1})(f(x))|\d x.
\end{aligned}
\end{equation}

First note that since $E[\rho_k]$ are uniformly bounded above, by Proposition \ref{prop:E-W1}, so is the sequence of first moments $\int_M \theta_x \rho_k(x) \d x$. This fact and the weak convergence $\rho_k \rightharpoonup \rho_0$ imply that the first moment of $\rho_0$ is bounded, i.e.,
\begin{equation}
\label{eqn:rho0-mnt}
\int_M \theta_x\rho_0(x)\d x\leq L,
\end{equation}
for some $L>0$. 

From the triangle inequality we get
\begin{align*}
\iint_{M\times M}\dist(x, y)\rho_0(x)\rho_0(y)\d x\d y & \leq \iint_{M\times M}(\theta_x+\theta_y)\rho_0(x)\rho_0(y)\d x\d y \\
& = 2\int_M\theta_x\rho_0(x)\d x,
\end{align*}
and hence, by using \eqref{eqn:rho0-mnt}, 
\[
0 \leq \iint_{M\times M}\dist(x, y)\rho_0(x)\rho_0(y)\d x\d y \leq 2L.
\]
On the other hand, by Theorem \ref{lemma:Chavel-thms} and \eqref{eqn:est-sinh}, we can estimate
\begin{equation}
\label{eqn:logJ-theta}
\log|J(f^{-1})(f(x))|\leq \sqrt{c_m}(\dm-1)\theta_x,
\end{equation}
and hence,
\begin{align*}
\int_M \rho_0(x)\log |J(f^{-1})(f(x))|\d x &\leq \int_M \sqrt{c_m}(\dm-1)\theta_x\rho_0(x)\d x.
\end{align*}
Together with $\log |J(f^{-1})(f(x))| \geq 0$ and  \eqref{eqn:rho0-mnt}, this leads to
\[
0 \leq \int_M \rho_0(x)\log |J(f^{-1})(f(x))|\d x \leq \sqrt{c_m}(\dm-1)L.
\]

Fix $\epsilon>0$ arbitrary small. By the arguments above,
\[
\Biggl | \sqrt{c_m}(\dm-1)\iint_{M\times M}\dist(x, y)\rho_0(x)\rho_0(y)\d x\d y- \int_M \rho_0(x)\log |J(f^{-1})(f(x))|\d x \Biggr |<\infty,
\]
and hence, there exists $\tilde{R}$ such that
\[
\sqrt{c_m}(\dm-1)\iint_{(M\times M)\backslash(\overline{B_R(\p)}\times\overline{B_R(\p)})}\dist(x, y)\rho_0(x)\rho_0(y)\d x\d y- \int_{M\backslash\overline{B_R(\p)}}\rho_0(x)\log |J(f^{-1})(f(x))|\d x<\epsilon
\]
for any $R>\tilde{R}$. Then, by the weak convergence $\rho_k \otimes \rho_k \rightharpoonup \rho_0 \otimes \rho_0$, we find 
\begin{equation}
\label{eqn:est-diff}
\begin{aligned}
&\sqrt{c_m}(\dm-1)\iint_{M\times M}\dist(x, y)\rho_0(x)\rho_0(y)\d x\d y- \int_M \rho_0(x)\log |J(f^{-1})(f(x))|\d x\\
& \quad \leq \sqrt{c_m}(\dm-1)\iint_{\overline{B_R(\p)}\times\overline{B_R(\p)}}\dist(x, y)\rho_0(x)\rho_0(y)\d x\d y- \int_{\overline{B_R(\p)}}\rho_0(x)\log |J(f^{-1})(f(x))|\d x+\epsilon\\
& \quad \leq \lim_{k\to\infty}\left(\sqrt{c_m}(\dm-1)\iint_{\overline{B_R(\p)}\times\overline{B_R(\p)}}\dist(x, y)\rho_k(x)\rho_k(y)\d x\d y- \int_{\overline{B_R(\p)}}\rho_k(x)\log |J(f^{-1})(f(x))|\d x\right)+\epsilon,
\end{aligned}
\end{equation}
for any $R > \tilde{R}$.

From  \eqref{eqn:dxy-ineq} and \eqref{eqn:int-cos}, since $\rho_k\in\mathcal{P}_\p(M)$, we can estimate
\begin{align*}
\int_{M}\dist(x, y)\rho_k(y)\d y&\geq \int_{M}\left(\theta_x^2+\theta_y^2-2\theta_x\theta_y\cos\angle(x\p y)\right)^{1/2}\rho_k(y)\d y\\
&\geq \int_{M}(\theta_x-\theta_y\cos\angle(x\p y))\rho_k(y)\d y\\[2pt]
&=\theta_x,
\end{align*}
for any $x\in M$ fixed . Using this fact, for any measurable set $S\subset M$, we have
\begin{equation}
\begin{aligned}
&\sqrt{c_m}(\dm-1)\iint_{(M\times M)\backslash (S\times S)} \dist(x, y)\rho_k(x)\rho_k(y)\d x \d y-\int_{M\backslash S} \rho_k(x)\log |J(f^{-1})(f(x))|\d x\\
& \quad \geq \sqrt{c_m}(\dm-1)\iint_{ (M\backslash S)\times M} \dist(x, y)\rho_k(x)\rho_k(y)\d x \d y-\int_{M\backslash S} \rho_k(x)\log |J(f^{-1})(f(x))|\d x\\
&\quad \geq \sqrt{c_m}(\dm-1)\int_{M\backslash S} \theta_x\rho_k(x)\d x-\int_{M\backslash S} \rho_k(x)\log |J(f^{-1})(f(x))|\d x\\
&\quad \geq 0,
\end{aligned}
\end{equation}
where for the last inequality we used \eqref{eqn:logJ-theta}.

By the inequality above with $S=\overline{B_R(\p)}$, we infer
\begin{align*}
& \sqrt{c_m}(\dm-1)\iint_{\overline{B_R(\p)}\times \overline{B_R(\p)}} \dist(x, y)\rho_k(x)\rho_k(y)\d x \d y-\int_{\overline{B_R(\p)}} \rho_k(x)\log |J(f^{-1})(f(x))|\d x \\
& \hspace{2cm} \leq \sqrt{c_m}(\dm-1)\iint_{M\times M} \dist(x, y)\rho_k(x)\rho_k(y)\d x \d y-\int_{M} \rho_k(x)\log |J(f^{-1})(f(x))|\d x,
\end{align*}
for all $k \geq 1$. This result allows us to continue estimating \eqref{eqn:est-diff} as
\begin{align*}
&\sqrt{c_m}(\dm-1)\iint_{M\times M}\dist(x, y)\rho_0(x)\rho_0(y)\d x\d y- \int_M \rho_0(x)\log |J(f^{-1})(f(x))|\d x\\
\leq&\liminf_{k\to\infty}\left(\sqrt{c_m}(\dm-1)\iint_{M\times M}\dist(x, y)\rho_k(x)\rho_k(y)\d x\d y- \int_{M}\rho_k(x)\log |J(f^{-1})(f(x))|\d x\right)+\epsilon.
\end{align*}
Finally, since $\epsilon>0$ is arbitrary, we get the desired result \eqref{eqn:liminf-part3}.
\end{proof}

\begin{remark}
\label{rmk:lsc-decouple}
\normalfont
In Proposition \ref{prop:lsc}, where $h$ satisfies the superlinear growth condition \eqref{eqn:h-hyp}, we showed that each component of the energy, the entropy and the interaction energy, respectively, is lower semi-continuous. In particular, the lower semi-continuity of the entropy was enabled by Lemma \ref{lem:Jint-cont}, that provides the continuity of the term involving the Jacobian of the exponential map. Under the weaker assumption \eqref{eqn:assumptions-h}, we were not able to separate the two components in the proof of Proposition \ref{prop:lsc-homog}. In this case (e.g., see Part 3 of the proof), we had to couple the two components in order to show the lower semi-continuity of the energy.
\end{remark}

Now we present the proof of Theorem \ref{them:existence}. 
\medskip

{\em Proof of Theorem \ref{them:existence}.} We will follow the general strategy from the proof of Theorem \ref{thm:exist-gen}. Step 1 in the proof of Theorem \ref{thm:exist-gen} follows identically and one can conclude that
\[
E_0:=\inf_{\rho\in \mathcal{P}_\p(M)}E[\rho] = \inf_{\rho\in \overline{B_{R}(\delta_\p)}\cap\mathcal{P}_\p(M)}E[\rho],
\]
for some $R>0$ fixed.

Take a minimizing sequence $\{\rho_k\}_{k\in\mathbb{N}}$ of $E[\rho]$ in $\overline{B_R(\delta_\p)}\cap \mathcal{P}_\p(M)$, i.e.,
\[
\{\rho_k\}\subset \overline{B_R(\delta_\p)}\cap \mathcal{P}_\p(M), \qquad \lim_{k\to\infty}E[\rho_k]=E_0.
\]
Assume without loss of generality that $E[\rho_k]$ is non-increasing, which implies that $E[\rho_k]$ is uniformly bounded from above. Since $\displaystyle \overline{B_R(\delta_\p)}$ is tight, there exists a subsequence $\{\rho_{k_l}\}_{l\in\mathbb{N}}$ of $\{\rho_k\}_{k\in\mathbb{N}}$ which converges weakly as measures to $\rho_0\in \overline{B_R(\delta_\p)}$ as $l \to \infty$. Also, from Proposition \ref{prop:lsc-homog}, we have
\begin{equation}
\label{eqn:limErho-kl-h}
\lim_{l\to\infty}E[\rho_{k_l}]\geq E[\rho_0].
\end{equation}

Denote by $\p'$ the centre of mass of $\rho_0$. Since the manifold $M$ is homogeneous, there exists an isometry $i:M\to M$ such that $i(\p') = \p$. Define $\bar{\rho}_0$ to be the pushforward (as measures) of $\rho_0$ by $i$, i.e., $\bar{\rho}_0 =i_\#\rho_0$. 
Then the centre of mass of $\bar{\rho}_0$ is $\p$ and hence, $\bar{\rho}_0\in \mathcal{P}_\p(M)$. 

As noted in Section \ref{sect:prelim}, the energy is invariant to isometries, and therefore $E[\rho_0]=E[\bar{\rho}_0]$. Using this fact along with $\rho_{k_l}$ being a minimizing sequence, and \eqref{eqn:limErho-kl-h}, we then find
\[
E_0=\lim_{l\to\infty}E[\rho_{k_l}]\geq E[\rho_0]=E[\bar{\rho}_0]\geq \inf_{\rho\in \mathcal{P}_\p(M)}E[\rho]=E_0.
\]
Finally, we infer $E[\bar{\rho}_0]= E_0$, and conclude that $\bar{\rho}_0$ is a global energy minimizer of the energy in $\mathcal{P}_\p(M)$. 

\hspace {14cm} $\qed$


\appendix
\section{Proof of Lemma \ref{Lem6.1}}
\label{appendix:Lem6.1}
Consider the Rauch comparison theorem in the same setup as that used to show \eqref{eqn:dist-ineq} (see paragraph under equation \eqref{eqn:dist-ineq}). Then, also using the Law of Cosines, we have
\begin{align}
\dist(x, y)^2 &\geq |\log_\p x- \log_\p y|^2 \nonumber \\[2pt]
& =  \theta_x^2+\theta_y^2-2\theta_x\theta_y\cos(\angle x\p y), \label{eqn:dxy-ineq}
\end{align}
which further yields
\begin{equation}
\label{eqn:loglog-ineq}
2\log_\p x\cdot\log_\p y=2\theta_x\theta_y\cos(\angle x\p y)\geq\theta_x^2+\theta_y^2-\dist(x, y)^2.
\end{equation}

Using the triangle inequality and some trivial calculations, we get
\begin{align}
\label{N-1-1}
\begin{aligned}
(\theta_x+\theta_y-\dist(x, y))^2-2\theta_x\theta_y &=\theta_x^2+\theta_y^2+\dist(x, y)^2-2\dist(x, y)(\theta_x+\theta_y)  \\
& \leq\theta_x^2+\theta_y^2+\dist(x, y)^2-2\dist(x, y)^2 \\
&= \theta_x^2+\theta_y^2-\dist(x, y)^2.
\end{aligned}
\end{align}
Multiply \eqref{eqn:fixed-cm} by $\log_\p y \rho(y)$ and integrate, and use \eqref{eqn:loglog-ineq} and \eqref{N-1-1}, to get
\begin{align}
\label{eqn:ineq-int1}
\begin{aligned}
0&=\iint_{M\times M} 2 \log_\p x\cdot \log_\p y \rho(x)\rho(y)\d x\d y \\
&\geq \iint_{M \times M} \left((\theta_x+\theta_y-\dist(x, y))^2-2\theta_x\theta_y\right)\rho(x)\rho(y)\d x \d y \\
&=\iint_{M \times M} (\theta_x+\theta_y-\dist(x, y))^2\rho(x)\rho(y)\d x \d y-2\int_M \theta_x\rho(x)\d x\int_M\theta_y\rho(y)\d y.
\end{aligned}
\end{align}
Furthermore, by using the Cauchy-Schwarz inequality
\[
 \iint_{M \times M}(\theta_x+\theta_y-\dist(x, y))\rho(x)\rho(y) \d x \d y \leq \left( \iint_{M \times M} (\theta_x+\theta_y-\dist(x, y))^2\rho(x)\rho(y)\d x \d y \right)^{1/2},
\]
we find from \eqref{eqn:ineq-int1}:
\begin{align}
\label{eqn:ineq-int2}
\begin{aligned}
0 &\geq \left( \iint_{M \times M}(\theta_x+\theta_y-\dist(x, y))\rho(x)\rho(y) \d x \d y\right)^2-2\left(\int_M \theta_x\rho(x)\d x\right)^2\\
&=\left(2\int_M \theta_x\rho(x)\d x- \iint_{M \times M}\dist(x, y)\rho(x)\rho(y)\d x \d y\right)^2-2\left(\int_M \theta_x\rho(x)\d x\right)^2.
\end{aligned}
\end{align}

By the definition of the $1$-Wasserstein distance, write
\[
\calW_1(\rho, \delta_\p)=\int_M \theta_x \rho(x)\d x.
\]
Hence, by \eqref{eqn:ineq-int2}, we get
\begin{align*}
0&\geq \left(2\calW_1(\rho, \delta_\p)-\iint_{M \times M} \dist(x, y)\rho(x)\rho(y)\d x \d y\right)^2-2\calW_1(\rho, \delta_\p)^2\\
&=2\calW_1(\rho, \delta_\p)^2-4\calW_1(\rho, \delta_\p) \iint_{M \times M} \dist(x, y)\rho(x)\rho(y) \d x \d y+\left(\iint_{M \times M} \dist(x, y)\rho(x)\rho(y)\d x \d y\right)^2.
\end{align*}
Therefore, for any $\alpha>0$ and $0<\beta<2$, we have
\begin{align*}
& \alpha\left(\iint_{M \times M} \dist(x, y)\rho(x)\rho(y)\d x \d y\right)^2-\beta \calW_1(\rho, \delta_\p)^2 \\
& \quad \geq(2-\beta)\calW_1(\rho, \delta_\p)^2 -4\calW_1(\rho, \delta_\p) \iint_{M \times M} \dist(x, y)\rho(x)\rho(y)\d x \d y+(1+\alpha)\left(\iint_{M \times M} \dist(x, y)\rho(x)\rho(y)\d x \d y\right)^2.
\end{align*}
To make the right-hand-side a perfect square, take $\alpha$ and $\beta$ that satisfy
\[
(2-\beta)(1+\alpha)=4 \quad\Longleftrightarrow\quad \alpha=\frac{2+\beta}{2-\beta}.
\]
Consequently, it holds that
\[
\left(\frac{2+\beta}{2-\beta}\right)\left(\iint_{M \times M} \dist(x, y)\rho(x)\rho(y)\d x \d y\right)^2-\beta \calW_1(\rho, \delta_\p)^2\geq0,
\]
for any $0<\beta<2$. 

Finally, since $\frac{\beta(2-\beta)}{2+\beta}$ has a maximum at $\beta=2\sqrt{2}-2$, we find
\[
\iint_{M \times M} \dist(x, y)\rho(x)\rho(y)\d x \d y\geq (2-\sqrt{2})\calW_1(\rho, \delta_\p),
\]
which is the first of the desired double inequality.
The second inequality follows immediately from
\[
\iint_{M \times M} \dist(x, y)\rho(x)\rho(y)\d x \d y\leq \iint_{M \times M} (\theta_x+\theta_y)\rho(x)\rho(y) \d x \d y=2\calW_1(\rho, \delta_\p).
\]


\section{Proof of Lemma \ref{lemma:unif-int}}
\label{appendix:unif-int}
Let 
\[
\bar{\rho}_k(x)=\begin{cases}
0\quad&\text{if }\rho_k(x)\geq1,\\[2pt]
\rho_k(x)\quad&\text{if }0\leq \rho_k(x)<1,
\end{cases}
\]
and define
\[
m_k=\int_M \bar{\rho}_k(x)\d x.
\]
Also set
\[
\mu(x)=\frac{\theta_x^{\dm-1} e^{-\theta_x}}{\alpha|\partial B_{\theta_x}(\p)|},
\]
where $\alpha$ is a positive constant which normalizes $\mu$, i.e.,
\begin{align*}
\alpha&
=\int_0^\infty \theta_x^{\dm-1} e^{-\theta_x}\d\theta_x.
\end{align*}
Hence,
\[
\int_M \d\mu(x) = 1,
\]
where $\d\mu(x)=\mu(x)\d x$, and we can consider $\mu(x)$ as a probability measure. 

Define
\[
U_k(x)=\frac{\bar{\rho}_k(x)}{\mu(x)}.
\]
A simple calculation yields
\[
\int_M U_k(x)\d\mu(x)= \int_M \bar{\rho}_k(x)\d x=m_k,
\]
and by applying Jensen's inequality to the convex function $ x \in [0, \infty) \mapsto x \log x$, we then find
\begin{align*}
m_k\log m_k &= \left(\int_M U_k(x)\d\mu(x)\right)\log\left(\int_M U_k(x)\d\mu(x)\right) \\
& \leq \int_M U_k(x)\log U_k(x)\d\mu(x).
\end{align*}

Since $0\leq m_k\leq 1$, we have $m_k\log m_k\geq -\frac{1}{e}$, which combined with the inequality above it yields
\begin{align}
-\frac{1}{e}&\leq \int_M U_k(x)\log U_k(x) \d\mu(x) \nonumber \\
&=\int_M \bar{\rho}_k(x)(\log\bar{\rho}_k(x)-\log\mu(x))\d x \nonumber \\
&=\int_M\bar{\rho}_k(x)\log\bar{\rho}_k(x)\d x+\int_M \bar{\rho}_k(x)\left(\theta_x+\log\alpha+\log\left(\frac{|\partial B_{\theta_x}(\p)|}{\theta_x^{\dm-1}}\right)\right)\d x.
\label{eqn:ineq-oe-1}
\end{align}
Analogous to the volume comparison result (see Corollary \ref{cor:AV-bounds} part i)), the area comparison theorem provides that
\begin{equation}
\label{ineq:area-comp}
|\partial B_{\theta_x}(\p)| \leq \dm w(\dm)\left( \frac{\sinh(\sqrt{c_m}\theta_x)}{\sqrt{c_m}}\right)^{\dm-1},
\end{equation}
which used in \eqref{eqn:ineq-oe-1} it gives 
\begin{equation*}
\begin{aligned}
-\frac{1}{e}&\leq \int_M\bar{\rho}_k(x)\log\bar{\rho}_k(x)\d x \\
& \quad +\int_M \bar{\rho}_k(x)\left( \theta_x+\log\alpha+\log(\dm w(\dm))+(\dm-1)\log\left( \frac{\sinh(\sqrt{c_m}\theta_x)}{\sqrt{c_m}\theta_x}\right) \right)\d x.
\end{aligned}
\end{equation*}
From the inequality above, and the fact that $\bar{\rho}_k \leq \rho_k$, we then get
\begin{equation}
\label{eqn:ineq-oe-2}
\begin{aligned}
-\int_M \bar{\rho}_k(x)\log\bar{\rho}_k(x)\d x &\leq \frac{1}{e}+\int_M \theta_x \rho_k(x) \d x \\
& \quad +(\dm-1)\int_M  \log\left( \frac{\sinh(\sqrt{c_m}\theta_x)}{\sqrt{c_m}\theta_x}\right) \rho_k(x) \d x+\max(0, \log(\alpha\dm w(\dm))).
\end{aligned}
\end{equation}

Now, using \eqref{eqn:ineq-oe-2}, one can further derive
\begin{align}
&\int_M \rho_k(x)|\log \rho_k(x)|\d x\\
&\quad =-\int_{\{x:\rho_k(x)\leq 1\}}\rho_k(x)\log\rho_k(x)\d x+\int_{\{x:\rho_k(x)> 1\}}\rho_k(x)\log\rho_k(x)\d x\\
&\quad =-2\int_{\{x:\rho_k(x)\leq 1\}}\rho_k(x)\log\rho_k(x)\d x+\int_M\rho_k(x)\log\rho_k(x)\d x\\
&\quad \leq 2\left(
\frac{1}{e}+\int_M\theta_x \rho_k(x)\d x+(\dm-1)\int_M \log\left( \frac{\sinh(\sqrt{c_m}\theta_x)}{\sqrt{c_m}\theta_x}\right) \rho_k(x) \d x+\max(0, \log(\alpha\dm w(\dm)))
\right) \\
&\quad \quad +\int_M\rho_k(x)\log\rho_k(x)\d x.
\end{align}
We note that by \eqref{eqn:est-sinh}, since $\int_M\theta_x \rho_k(x)\d x$ is uniformly bounded above in $k$, so is \\
 $\int_M \log\left( \frac{\sinh(\sqrt{c_m}\theta_x)}{\sqrt{c_m}\theta_x}\right) \rho_k(x) \d x$. Hence, all the three integral terms in the right-hand-side above are uniformly bounded in $k$, which implies that
\[
\int \rho_k(x)|\log\rho_k(x)| \d x
\]
is also uniformly bounded above in $k \in\mathbb{N}$.

For any $Q>1$, it holds that
\[
\int _M \rho_k(x)|\log\rho_k(x)|\d x\geq \int_{\{x:\rho_k(x)\geq Q\}}\rho_k(x)|\log\rho_k(x)|\d x\geq \int_{\{x:\rho_k(x)\geq Q\}}\rho_k(x)\log Q \, \d x,
\]
which implies
\[
\int_{\{x:\rho_k(x)\geq Q\}}\rho_k(x)\d x\leq \frac{1}{\log Q}\int_M \rho_k(x)|\log\rho_k(x)| \d x.
\]
Since $\int_M \rho_k(x)|\log\rho_k(x)|\d x$ is uniformly bounded in $k\in\mathbb{N}$, we have $\sup_{k\in\mathbb{N}}\int \rho_k(x)|\log\rho_k(x)| \d x<\infty$. This yields
\[
0\leq \lim_{Q\to\infty}\sup_{k\in\mathbb{N}}\int_{\{x:\rho_k(x)\geq Q\}}\rho_k(x)\d x\leq\lim_{Q\to\infty} \frac{1}{\log Q}\left(\sup_{k\in\mathbb{N}}\int_M \rho_k(x)|\log\rho_k(x)|\d x\right)=0.
\]
Finally, since $\sup_{k\in\mathbb{N}}\int_{\{x:\rho_k(x)\geq Q\}}\rho_k(x)\d x$ is non-increasing in $Q$, we also have
\[
\inf_{Q\geq0} \sup_{k\in\mathbb{N}} \int_{\{x:\rho_k(x)\geq Q\}}\rho_k(x)\d x=0,
\]
which implies that $\{\rho_k\}_{k\in\mathbb{N}}$ is uniformly integrable. 


\section{Part 2 of proof of Proposition \ref{prop:lsc-homog}}
\label{appendix:part2}
We will follow closely part 2 of the proof of Proposition \ref{prop:lsc}. When the arguments there extend immediately to this context, we will simply mention so.

Define $C$ by
\[
C=\begin{cases}
0&\quad\text{when }\lim_{\theta\to0+}h(\theta)\leq -1,\\[2pt]
-\lim_{\theta\to 0+}h(\theta)-1&\quad\text{otherwise}.
\end{cases}
\]

{\em Behaviour near $0$.} The behaviour near $0$ of $h(\theta) -2(\dm-1)\sqrt{c_m} \theta + C$ is the same as the behaviour at $0$ of $h(\theta) + C$. Also, by the definition of $C$, we have
\begin{equation}
\label{eqn:leq-m1}
\lim_{\theta\to0+}(h(\theta)-2(\dm-1)\sqrt{c_m}\theta+C)\leq -1.
\end{equation}
This allows all the arguments for the behaviour near $0$ of $h(\theta)+C$ that we made in part 2 of the proof of Proposition \ref{prop:lsc}, to apply to $h(\theta) -2(\dm-1)\sqrt{c_m} \theta + C$ as well. Simply replacing $h(\theta)+C$ by $h(\theta) -2(\dm-1)\sqrt{c_m} \theta + C$ in these arguments, we infer that for any $\epsilon>0$ arbitrary fixed, there exists sufficiently small $\delta>0$ such that (see \eqref{hsmalldist}):
\begin{equation}
\label{eqn:iint-delta-h}
\Biggl | \iint_{\dist(x, y)<\delta}\left(h(\dist(x, y))-2(\dm-1) \sqrt{c_m} \dist(x, y)+C\right)\rho_k(x)\rho_k(y)\d x\d y\Biggr |<\epsilon, \qquad \text{for all } k \geq 1.
\end{equation}
Also (see \eqref{eqn:leq-m1}), $\delta$ can be chosen such that $h(\theta)-2(\dm-1)\sqrt{c_m}\theta+C <0$ for $0<\theta<\delta$.
\smallskip

{\em Behaviour away from $0$.}  \underline{Claim 1}:  For $\delta>0$ fixed (cf. \eqref{eqn:iint-delta-h}), 
\[
\iint_{\dist(x, y)\geq\delta} (h(\dist(x, y))-2(\dm-1)\sqrt{c_m}\dist(x, y)+C)\rho_k(x)\rho_k(y)\d x \d y
\]
is uniformly bounded above, with an upper bound independent of $\epsilon$.

To show this, we first note that although $h$ satisfies the weaker condition \eqref{eqn:assumptions-h}, we can prove \eqref{x-4} similarly.  Hence, by  Proposition \ref{prop:E-W1} and Lemma \ref{l33} we infer that $\int_M \theta_x \rho_k(x) \d x$ and $\int_M \rho_k(x) \log\rho_k(x) \d x$ are uniformly bounded. Furthermore, Lemma \ref{lemma:unif-int} applies, and from its proof we get the uniform boundedness of $\int_M \rho_k(x)|\log\rho_k(x)|\d x$. 

From an immediate calculation, we derive
\begin{align*}
&\frac{1}{2}\iint_{M\times M} (h(\dist(x, y))-2(\dm-1)\sqrt{c_m}\dist(x, y))\rho_k(x)\rho_k(y)\d x \d y\\
& \quad = E[\rho_k]-\int_M \rho_k(x)\log\rho_k(x)\d x-\sqrt{c_m}(\dm-1)\iint_{M\times M} \dist(x, y)\rho_k(x)\rho_k(y)\d x \d y\\
& \quad \leq E[\rho_k]+\int_M \rho_k(x)|\log\rho_k(x)|\d x.
\end{align*}
Hence, since $E[\rho_k]$ and $\int\rho_k(x)|\log\rho_k(x)|\d x$ are uniformly bounded above in $k$, so is 
\[
\iint_{M\times M} (h(\dist(x, y))-2(\dm-1)\sqrt{c_m}\dist(x, y))\rho_k(x)\rho_k(y)\d x \d y.
\]
Using this property, the fact that $\iint_{M \times M} C \rho_k(x) \rho_k(y) \d x \d y = C$, and \eqref{eqn:iint-delta-h} (use $\epsilon <1$), we can then infer Claim 1. 

Define $U$ (independent of $\epsilon$) such that 
\[
\iint_{\dist(x, y)\geq \delta} (h(\dist(x, y))-2(\dm-1)\sqrt{c_m}\dist(x, y)+C)\rho_k(x)\rho_k(y)\d x \d y \leq U, \qquad \text{ for all } k \geq 1.
\]
\smallskip

\underline{Claim 2}: With $\delta>0$ fixed as above,
\begin{equation}
\label{eqn:iint-ldelta-h}
\iint_{\dist(x, y) \geq \delta} (h(\dist(x, y))-2(\dm-1)\sqrt{c_m}\dist(x, y) + C)\rho_0(x)\rho_0(y)\d x \d y<\infty.
\end{equation}
This follows from Claim 1 in exactly the same way we showed Claim 2 in part 2 of the proof of Proposition \ref{prop:lsc}. For this purpose, we note that 
\[
\lim_{\theta\to\infty}(h(\theta)-2(\dm-1)\sqrt{c_m}\theta + C) = \infty,
\]
which follows from the behaviour at infinity of $h$ (the second condition in \eqref{eqn:assumptions-h}) by writing
\[
h(\theta)-2(\dm-1)\sqrt{c_m}\theta + C = \h(\theta)-A_2 \sqrt{c_m} \theta + C + (A_2 - 2(\dm-1))\sqrt{c_m}\theta,
\]
where $A_2>2 (\dm -1)$. Consequently, there exists $\tilde{R}$ such that 
\begin{equation}
\label{eqn:expr-pos}
h(\theta)-2(\dm-1)\sqrt{c_m}\theta +C>0, \qquad \text{ for all } \theta\geq \tilde{R}.
\end{equation}
Claim 2 will be shown then by a proof by contradiction argument, exactly as for Claim 2 in part 2 of the proof of Proposition \ref{prop:lsc}. We leave the details to the reader.

Now, from \eqref{eqn:iint-ldelta-h}, there exists $R_0>\tilde{R}$ such that
\[
0 \leq \iint_{\dist(x, y)>R_0}(h(\dist(x, y))-2(\dm-1)\sqrt{c_m}\dist(x, y)+C )\rho_0(x)\rho_0(y)\d x \d y<\epsilon.
\]
Then, also using $h(\theta)-2(\dm-1)\sqrt{c_m}\theta+C <0$ for $0<\theta<\delta$, and the weak convergence $\rho_k \otimes \rho_k \rightharpoonup \rho_0 \otimes \rho_0$, we derive
\begin{align*}
& \iint_{M\times M} (h(\dist(x, y))-2(\dm-1)\sqrt{c_m}\dist(x, y) + C)\rho_0(x)\rho_0(y)\d x \d y\\
& \hspace{1cm} =  \iint_{\dist(x, y)<\delta} (h(\dist(x, y))-2(\dm-1)\sqrt{c_m}\dist(x, y) + C)\rho_0(x)\rho_0(y)\d x \d y \\
& \hspace{1cm} \quad + \iint_{\delta\leq \dist(x, y)\leq R_0} (h(\dist(x, y))-2(\dm-1)\sqrt{c_m}\dist(x, y) + C)\rho_0(x)\rho_0(y)\d x \d y\\
& \hspace{1cm} \quad +\iint_{R_0< \dist(x, y)} (h(\dist(x, y))-2(\dm-1)\sqrt{c_m}\dist(x, y) +C)\rho_0(x)\rho_0(y)\d x \d y\\
& \hspace{1cm} < 0+\lim_{k\to\infty}\iint_{\delta\leq \dist(x, y)\leq R_0} (h(\dist(x, y))-2(\dm-1)\sqrt{c_m}\dist(x, y))\rho_k(x)\rho_k(y)\d x \d y+\epsilon.
\end{align*}

Finally, using \eqref{eqn:iint-delta-h} and \eqref{eqn:expr-pos} (note that $R_0> \tilde{R}$) in the above, we get
\begin{align*}
&\iint_{M\times M} (h(\dist(x, y))-2(\dm-1)\sqrt{c_m}\dist(x, y) + C)\rho_0(x)\rho_0(y)\d x \d y\\
& \hspace{1cm}  \leq  \liminf_{k\to\infty} \Biggl ( \iint_{\dist(x, y)\leq R_0} (h(\dist(x, y))-2(\dm-1)\sqrt{c_m}\dist(x, y) + C)\rho_k(x)\rho_k(y)\d x \d y + \epsilon \Biggr ) + \epsilon \\
& \hspace{1cm}  \leq \liminf_{k\to\infty}\iint_{M \times M} (h(\dist(x, y))-2(\dm-1)\sqrt{c_m}\dist(x, y) +C)\rho_k(x)\rho_k(y)\d x \d y+2\epsilon.
\end{align*}
Since $\epsilon$ is arbitrary, and 
\[
\iint_{M \times M} C\rho_0(x)\rho_0(y)\d x\d y= \iint_{M \times M} C\rho_k(x)\rho_k(y)\d x\d y = C,
\]
one can then reach \eqref{eqn:liminf-part2}, which is the desired result.
\medskip

{\large \bf Data availability}
\medskip

The data that supports the findings of this study is available from the corresponding author upon reasonable request.

\bibliographystyle{abbrv}
\def\url#1{}
\bibliography{lit.bib}

\end{document}